\documentclass{amsart}

\usepackage{amsmath, amsthm}
\usepackage{amssymb, latexsym}
\usepackage{amsfonts}

\numberwithin{equation}{section}

\newtheorem{lemma}[equation]{Lemma}

\newtheorem{conjecture}[equation]{Conjecture}

\newtheorem{theorem}[equation]{Theorem}
\newtheorem{corollary}[equation]{Corollary}

\newtheorem{question}[equation]{Question}
\newtheorem{proposition}[equation]{Proposition}

\newtheorem*{A}{Theorem A}

\newtheorem*{C'}{Theorem C'}
\newtheorem*{C''}{Theorem C''}

\newtheorem*{A'}{Theorem B}
\newtheorem*{A''}{Theorem C}
\newtheorem*{3}{Theorem D}
\newtheorem*{4}{Theorem E}
\newtheorem*{5}{Theorem F}
\newtheorem*{6}{Theorem G}

\newcommand{\beql}[1]{\begin{equation}\label{#1}}
\newcommand{\eeq} {\end{equation}}

\font\Bbb=msbm10

\def\F{\hbox{\Bbb F}}

\def\FF{\hbox{\rm F}}

\def\aut{\mathrm{Aut}}

\def\rad{\mathrm{Rad}}
\def\hom{\mathrm{Hom}}
\def\Hom{\mathrm{Hom}}

\def\FFF{{\mathrm{F}_4}}
\def\EE{{\mathrm{E}}}
\def\G{{\mathrm{G}}}

      \def\PSL{{\rm PSL}}
      \def\SL{{\rm SL}}

      \def\SU{{\rm SU}}
      \def\GU{{\rm GU}}
      \def\PSU{{\rm PSU}}

      \def\Sp{{\rm Sp}}
      
      \def\GL{{\rm GL}}

      \def\POmega{{\rm P\hbox{$\Omega$}}}

      \def\Re{\mathrm{{^2}G_2}}
      \def\Sz{{\rm Sz}}

\begin{document}

\title[Presentations of Simple Groups]
{Presentations of Finite Simple Groups:  \\
Profinite and Cohomological Approaches}
\author[Guralnick et al]{Robert Guralnick}
\address
{Department of Mathematics, University of Southern California,
 Los Angeles CA 90089-2532, USA}
\email{guralnic@usc.edu}
\author[]{William M. Kantor}
\address{Department of Mathematics, University of Oregon,
Eugene, OR  97403-1222, USA}
\email{kantor@math.uoregon.edu}
\author[]{Martin Kassabov}
\address{Department of Mathematics, Cornell University, Ithaca, NY 14853, USA}
\email{kassabov@math.cornell.edu}
\author[]{Alexander Lubotzky}
\address{Department of Mathematics, Hebrew University, 
Givat Ram, Jerusalem 91904, Israel}
\email{alexlub@math.huji.ac.il}
\date{\today}
\thanks{The authors were partially supported by
        NSF grants DMS 0140578, DMS 0242983, DMS 0600244
          and   DMS 0354731.  The  authors are grateful for
the support and hospitality of   the Institute for Advanced Study, where
this research was carried out.
The research by the fourth author also was supported by the ISF, 
the Ambrose Monell Foundation and the Ellentuck Fund.  We   also 
thank the two referees for their careful reading and comments.}

\begin{abstract}
We prove the
following three closely related results:
\begin{enumerate}
\item Every finite simple group 
$G$ has a profinite presentation with 2 generators and at
most $18$ relations.
\item If $G$ is a finite simple group, $F$  a field and
$M$  an $FG$-module, then
$ \dim H^2(G,M) \le (17.5) \dim M$.
\item If $G$ is a finite group, $F$ a field  and $M$  an irreducible
faithful  $FG$-module, then $\dim H^2(G,M) \le (18.5) \dim M$.
\end{enumerate}
\end{abstract}

\maketitle

\centerline{Dedicated to our friend and colleague Avinoam Mann }

\tableofcontents

\section{Introduction}

The main goal of this paper is to prove the following three
results which are essentially equivalent to each other.  Recall
that a quasisimple group is one that is perfect and  
simple modulo its center.  Note that the last theorem
is about all finite groups. 

\begin{A} Every finite quasisimple group $G$ has a profinite
presentation with $2$ generators and at most $18$ relations.
\end{A}

\begin{A'} If $G$ is a finite quasisimple group, $F$  a field and
$M$  an $FG$-module, then   
$ \dim H^2(G,M)\leq (17.5) \dim M $.
\end{A'}

\begin{A''} If $G$ is a finite group, $F$  a field and $M$ an   irreducible
faithful $FG$-module, then $\dim H^2(G,M) \le (18.5) \dim M$.
\end{A''}

All three theorems depend on the classification of finite simple groups.
One could prove Theorems A and B independently of the classification
for the known simple groups.

We  abuse notation somewhat and say that an $FG$-module is faithful
if $G$ acts faithfully on $M$.  We call $M$   a trivial $G$-module
if it is $1$-dimensional and $G$ acts trivially on $M$.

In \cite{pap1}, the predecessor of this article, we showed
that  every finite non-abelian simple group, with the
possible exception of the family  ${^2}\G_2(3^{2k+1})$, has a bounded short
presentation (with at most $1000$ relations  -- short being
defined in terms of the sums of the lengths of the relations). 
We deduced results similar to the first two theorems above
but with larger constants.  In \cite{pap3}, we show that 
every finite simple
group (with the possible exception of ${^2}\G_2(3^{2k+1})$) has
a presentation with at most $2$ generators and $100$ relations.

In many cases, the results proved here and in
\cite{pap3}  are  much better, e.g, for $A_n$ and $S_n$,
we produce presentations with $3$ generators and at most $7$
relations \cite{pap3}.  Here we give still better results for these
groups in the profinite case -- there are profinite presentations
with $2$ generators and at most $4$ relations.

We believe that with more effort (and some additional
ideas) the constants in  these  three theorems
 may be dropped to $4$, $2$ and $1/2$
respectively.  

One of the methods used in this paper is  possibly of 
as much interest  as  the 
results themselves.  We show how to combine cohomological and
profinite presentations arguments -- by going back and forth
between the two topics to deduce results on both.
The bridge between the two subjects is a formula given in \cite{lub2}
which states: If $G$ is a finite group and $\hat{r}(G)$ 
is the minimal number of relations in a profinite presentation
of $G$, then 
$$
\hat{r}(G)= \sup_p \sup_M \Big( \Big\lceil  \frac{\dim
H^2 (G, M) - \dim H^1 (G, M)}{\dim M} \Big\rceil  + d(G) - \xi_M
\Big),  \eqno(1.1)
$$  
 where $d(G)$ is the minimum number of  generators for
$G$,
$p$ runs over all primes, $M$ runs over all irreducible
${\F}_pG$-modules, and $\xi_M = 0$ if $M$ is the trivial
module and 1 if not.
By  \cite{gurhoff}, if $G$ is a quasisimple finite group, then 
for every $\mathbb{F}_pG$ module $M$,
$$
\dim \ H^1(G,M) \leq  (1/2)\dim \ M. \eqno(1.2)
$$
Set
$$
h'_p(G)=\max_{M}  \frac{\dim H^2(G,M)}{\dim M}, \ \mathrm{and} \  
      h'(G) = \max_p h'_p(G),  \eqno(1.3)
$$
where $M$ ranges over nontrivial irreducible $\F_pG$-modules.
  If $G$ is a finite quasisimple
group, then $d(G) \leq 2$ \cite[Theorem B]{ag}
and $\dim H^2(G,\F_p) \le 2$ \cite[pp. 312--313]{gls3}) and so
$$
\max \{2, \lceil h'(G) + 1/2 \rceil\} \le \hat{r}(G) \le  
\max\{4, \lceil h'(G) + 1 \rceil \}.   \eqno(1.4)
$$
This explains how Theorems A and B are related and are
essentially equivalent.   We see in Section \ref{faithful section}
that Theorem B implies Theorem C.  On the other hand,
the   bound for Schur multipliers for finite simple groups
 and Theorem C implies a version
of Theorem B.

We also define
$$
h(G)= \max_{M,p}  \frac{\dim H^2(G,M)}{\dim M},   \eqno(1.5)
$$
 where $M$ ranges over all $\F_pG$-modules.  

We now give an outline of the paper. 
After some preparation
in Sections \ref{prelim1}, \ref{covering}, and \ref{faithful section},
we show in Sections \ref{alternating}, 
\ref{SLlow}  and \ref{low rank}, 
respectively,  that:

\begin{3} For every $n$, $h(A_n) < 3$  and  $h(S_n) < 3$
and $\hat{r}(A_n)$ and $\hat{r}(S_n) \le 4$.
\end{3}

\begin{4}  For every prime power $q$ and  $2 \leq n \leq 4, \  h(\SL(n,q)) 
\leq 2$. 
\end{4}

\begin{5}   $\max \{h(G), \hat{r}(G)\} \leq 6$
for each rank $2$ quasisimple finite group $G$ of Lie type, 

 \end{5}

In fact, the results are more precise -- see sections \ref{alternating}, 
\ref{SLlow}  and \ref{low rank} for details.

From (1.4) we see that Theorems D, E and F  imply
that all the groups in those theorems have profinite presentations
with a small number of relations.   
In sections \ref{SLgen} and \ref{classical}, 
we repeat our   ``gluing''
arguments from \cite[\S 6.2]{pap1}  to show how to deduce from these cases the
existence of bounded (profinite) presentations for
all the quasisimple finite groups of Lie type.  In fact, this time 
the proof is easier and the
result is stronger as we do not insist of having a short
presentation as we did in \cite{pap1}; we count only the number of
relations but not their length.
Moreover,  Lemma \ref{saving}
gives an interesting method for saving relations which
seems to be new (the analog is unlikely to work for discrete presentations).
In Section \ref{sporadic}, we discuss the sporadic simple groups.   
If a Sylow $p$-subgroup has order at most $p^{m}$, one can use
the main result of \cite{holt} to deduce the bound 
$h'_p(G) \le 2m$.  In many sporadic cases, discrete presentations for the  
groups  are known
\cite{rwilson} and the results follow.  There are not too many
additional cases to consider.

This completes the outline of the proof of Theorem A.  Applying  (1.4) in
the reverse direction we deduce Theorem B (at least for $\F_p$ -- however,
changing the base field does not change the ratio $\dim H^2(G,M)/\dim M$-- see
Lemma \ref{extension of scalars} and the discussion following it).
 In Section \ref{faithful section},
we prove Theorem \ref{faithful} which shows that 
Theorem B implies Theorem C.

Holt \cite{holt} conjectured Theorem C for some constant  
 $C$. He proved  that
$$
\dim H^2(G,M) \le 2e_p(G) \dim M,
$$
for $M$ an irreducible faithful $G$-module,  where 
$p^{e_p(G)}$ is the order of a Sylow
$p$-subgroup of $G$.   Holt also reduced his proof to simple groups.
However,  he was proving a weaker result than we are aiming for, and
his reduction methods are not sufficient for our purposes.

As we have already noted in (1.2),  the analog of Theorem B for $H^1$ holds
with constant   $1/2$.   It is relatively easy to see
that this implies that 
the analog of Theorem C  for $H^1$ with constant $1/2$ is valid.   
We give examples
to show that the situation for higher cohomology groups is different
(see Section \ref{higher cohomology}).    In particular, the following
holds:

\begin{6}  Let $F$  be an
algebraically closed  field of characteristic 
$p > 0$ and   let $k$ be a positive integer.
\item  
There exists a sequence of finite groups $G_i, i \in \mathbb{N}$
and irreducible faithful $FG_i$ modules $M_i$ such that: 
\begin{enumerate}
\item $\lim_{i \rightarrow \infty}  \dim M_i = \infty$, 
\item $\dim H^k(G_i,M_i) \ge e (\dim M_i)^{k-1}$ for some
constant $e=e(k,p)> 0$, and
\item   if $k \ge 3$, then
$ \displaystyle { \lim_{i \rightarrow \infty} 
\frac{\dim H^k(G_i,M_i)}{\dim M_i} = \infty.}$
 \end{enumerate}
\end{6}

Thus the analog of Theorem C for $H^k$ with $k \ge 3$ does not
hold for any constant -- although it is still possible that an analog of
Theorem B holds.   This also shows that  $\dim H^2(G,M)$ 
can be arbitrarily
large for  faithful absolutely irreducible modules -- it is not
known whether this is possible for 
$H^1(G,M)$ under the same hypotheses.   We suspect that there is
an  upper  bound for $\dim H^k(G,M)$ of the same form as the lower
bound in (2) above.

Finally, in Section \ref{applications} we give some applications of the 
results in \cite{pap1}
and the current paper for general finite groups, as well as some questions.
An especially intriguing question is related to the fact that 
$\hat{r}(G) \le r(G)$, 
the minimal number of relations
required in any presentation of the group $G$.  As far as we know,
it is still not known whether for some finite group $G$, we can have
$\hat{r}(G) < r(G)$.

There is a long history of studying presentations of groups and,
in particular, the number and length of relations required
for finite groups.  Presentations of groups also rise  in connection with various
problems about counting isomorphism classes of groups.
Much of the work done recently on these questions (e.g., \cite{pap1}, 
\cite{korlu}, \cite{lub1}, and
\cite[Chapter 2]{lubsegal}) was motivated by the paper \cite{mann}
of Avinoam Mann.
We dedicate this paper to him on the occasion of his retirement.

\section{General Strategy and Notation} \label{strategy}

\subsection{Strategy} \hfil\break
  
We outline a  method for obtaining bounds of
the form $\dim H^2(G,M) \le C \dim M$ for some constant $C$.
Here $G$ is a finite group and $M$ is an $FG$-module with
$F$ a field of characteristic $p > 0$ (in characteristic zero,
$H^2(G,M)=0$ -- see Corollary \ref{coprime}).  
There are several techniques 
that we use to reduce the problem to smaller groups.

The first is to use the long exact sequence for cohomology 
(Lemma \ref{cohomology sequence}) to reduce to the case
that $M$ is irreducible.  Then we use Lemma \ref{extension of scalars}, 
which allows us to assume that we are over an algebraically closed
field and that $M$ is absolutely irreducible (occasionally, it is convenient
to use this in the reverse direction and assume that $M$ is finite and 
over $\F_p$  -- see the discussion after Lemma \ref{extension of scalars}).
We also  use the standard fact that $H^2(G,M)$ embeds in $H^2(H,M)$ whenever
$H \le G$ contains a Sylow $p$-subgroup of $G$ \cite[p. 91]{relmod}.
Typically, $M$ will no longer
be irreducible as an $FH$-module, but we can reduce to that case as above.

We  use these reductions often without comment.  

We use our results on low rank finite groups of Lie type and
the alternating groups to provide profinite presentations for
the larger rank finite groups of Lie type,  and so also bounds
for $H^2$ via (1.4).

\subsection{Notation}  \hfil\break

We use standard   terminology for finite groups.  In particular,
$\FF(G)$ is the Fitting group, $\FF^*(G)$ is the generalized Fitting
subgroup and $O_p(G)$ is the maximal normal $p$-subgroup of $G$.
A component is a subnormal quasisimple subgroup of $G$. 
The (central) product of all components of $G$ is denoted by $E(G)$. 
  Note that  $\EE(G)$ and $\FF(G)$ and commute. 
The generalized Fitting subgroup is $\FF^*(G) : = \EE(G)\FF(G)$.
 We let $C_t$ denote the
cyclic group of order $t$. 
See \cite{aschbook} for a general reference for finite group theory.
We also use \cite{gls3} as a general reference for properties of the
finite simple groups --  the Schur multipliers and outer automorphism
groups of all the simple groups are given there.

If $M$ is an $H$-module, $M^H$ is the set of $H$ fixed 
points on $M$ and $[H,M]$ is  the submodule 
generated by $\{hv-v| h \in H, v \in M\}$.  Note that $[H,M]$
is the smallest submodule $L$ of $M$ such that $H$ acts trivially 
on $M/L$.  If $V$ is a module for the subgroup $H$ of $G$, 
 $V_H^G$ is the induced module.

\section{Preliminaries on Cohomology} \label{prelim1}

Most of the results in this section are well known.  See \cite{baba}, 
\cite{brown}, \cite{maclane} and
\cite{relmod} for standard facts about group cohomology. 

We first state a result that is an easy  corollary of Wedderburn's 
theorem on finite division
rings. We give a somewhat different proof based
on Lang's theorem (of course, 
Wedderburn's theorem is a special case of Lang's
Theorem).   See also a result of Brauer \cite[19.3]{feit}
that is slightly weaker.

\begin{lemma} \label{endomorphism}
Let $K$ be a (possibly infinite)
field of characteristic $p > 0$, and
let $G$ be a finite group.   Let $V$ be an irreducible
$KG$-module.  
\begin{enumerate} 
\item There is a finite subfield $F$ of $K$
and an irreducible $FG$-module $W$ 
with $V \cong W \otimes_F K$. 
\item  $\mathrm{End}_{KG}(V)$ is  a field.
\end{enumerate}
\end{lemma}

\begin{proof} Clearly, (1) implies (2) by Wedderburn's 
Theorem and Schur's Lemma.  One can give a more direct
proof.  Let $F$ be a finite subfield of $K$.
Then $B:=KG \cong FG \otimes_F K$.  Thus,
$B/\rad(B)$ is a homomorphic image of
$(FG/\rad(FG)) \otimes_F K$.   By Wedderburn's Theorem,
$FG/\rad(FG)$ is a direct product of matrix rings
over fields, and so the same is true for $B/\rad(B)$.
Thus, $B/\mathrm{Ann}_B(V) \cong M_s(K')$ for some
extension field $K'/K$.  Since $K' \cong \mathrm{End}_G(V)$,
(2) follows.

We now prove (1).
  Set $n=\dim V$.
Let $\phi:G \rightarrow \GL(n,K)$ be the representation
determined by $V$.
Let $F$ be the subfield of $K$ generated
by the traces of elements of $\phi(g) \in G$ acting on $V$.
Let $q$ denote the cardinality of the finite field $F$.

 Let $\sigma$ denote the $q$th power map.  Note that $F$ is the fixed field of $\sigma$.
              Then $\sigma$  induces an endomorphisms of $KG$ and 
              so a twisted version of $V$, which denote by $V'$.  Let $L$ denote
the algebraic closure of $K$.   Since the character of
$V$ is defined over $F$, it follows that the characters of
$V$ and $V'$ are equal (and indeed, similarly for the
Brauer characters).  This implies that $V' \cong V$
as $KG$-modules (or equivalently as $LG$-modules).  Thus,
there exist $U \in \GL(n,K)$ with 
$U\phi(g)U^{-1} = \sigma(\phi(g))$ for all $g \in G$.
By Lang's theorem, $U = X^{-\sigma}X$ for some invertible
matrix $X$ (over $L$).   This implies that 
$\phi'(g):=X\phi(g)X^{-1}=\sigma(X\phi(g)X^{-1})$
defines a representation from $G$ into $\GL(n,F)$.
Let $U$ be the corresponding module.  Clearly, $V \cong U \otimes_F K$.
\end{proof}

We  state another  result about extensions of scalars.

\begin{lemma} \label{extension of scalars}
Let $G$ be a finite group and $F$ a field.  Let $M$ be an $FG$-module.
\begin{enumerate}
\item 
If $K$ is an extension field of $F$, then
$H^2(G,M) \otimes_F K$ and $H^2(G,M \otimes_F K)$
are naturally isomorphic, and in particular have
the same dimension.
\item 
If $M$ is irreducible and $F$ has positive characteristic, 
then $E:= \mathrm{End}_{G}(M)$ is a field,  $M$ is an absolutely
irreducible $EG$-module and $ \dim_F H^2(G,M) =[E:F] \dim_E H^2(G,M)$.
\end{enumerate} 
\end{lemma}

\begin{proof}  These results are well known.  See \cite[0.8]{brown} for the
first statement.  By Lemma \ref{endomorphism}, $E$ is a field.
Clearly, $M$ is an absolutely irreducible $EG$-module,
and so $H^2(G,M)$ is also a vector space over $E$.   The
last equality holds for any finite dimensional vector space over $E$.
\end{proof}

The previous result allows us to change fields in either direction.
If $F$ is algebraically closed of characteristic $p > 0$ and $M$
is an irreducible $FG$-module, then $M$ is defined over some finite field
$E$ -- i.e. there is an absolutely irreducible $EG$-module $V$ such
that $M = V \otimes_E F$ and we can compute the relevant ratios 
of dimensions over
either field.  Similarly,  if $M$ is an irreducible $FG$-module with
$F$ a finite field, then we can view $M$ an $EG$-module, 
where $E=\mathrm{End}_G(V)$,
and so assume that $M$ is absolutely irreducible.  Alternatively, we can
view $M$ as an $\F_pG$-module.

See \cite[III.6.1 and III.6.2]{brown} for the next two results.

\begin{lemma} \label{cohomology sequence}
Let $G$ be a group and $0 \rightarrow X \rightarrow Y 
\rightarrow Z \rightarrow 0$
a short exact sequence of $G$-modules. This induces an exact sequence:
$$
\begin{array}{clll}
0  & \rightarrow & H^0(G,X) &\rightarrow H^0(G,Y) \rightarrow H^0(G,Z)  
\rightarrow H^1(G,X) \rightarrow \cdots \\
& \rightarrow & H^{j-1}(G,Z) & \rightarrow H^j(G,X) \rightarrow H^j(G,Y) 
\rightarrow
H^j(G,Z) \rightarrow \cdots 
\end{array}
$$
In particular, $\dim H^j(G,Y) \le \dim H^j(G,X) + \dim H^j(G,Z)$ for any
integer $j \ge 0$.
\end{lemma}

\begin{lemma}[Shapiro's Lemma] \label{shapiro}
Let $G$ be a finite group and $H$ a subgroup
of $G$.  Let $V$ be an $FH$-module.  Then $H^j(H,V) \cong H^j(G,V_H^G)$
for any integer $j \ge 0$.
\end{lemma}

\begin{lemma} \label{cyclic}   
Let $G$ have a cyclic Sylow $p$-subgroup.
Let $F$ be a field of characteristic $p$. 
If $M$ is an indecomposable $FG$-module and $j$ is a non-negative
integer,  then $\dim H^j(G,M)  \le 1$.
\end{lemma} 

\begin{proof}  By a result of D. G. Higman (see \cite[3.6.4]{benson}),
$M$ is a direct summand of
$W_P^G$,  where $P$ is a  Sylow $p$-subgroup of $G$ and $W$ is
an $FP$-module.   Since $M$ is indecomposable,
we may assume that   $W$ is an indecomposable 
$P$-module.   
By Shapiro's Lemma (Lemma \ref{shapiro}), $H^j(G,M)$ is a summand of
$H^j(P,W)$.  So it suffices to assume that $G=P$ is
a cyclic $p$-group and $W$ is an indecomposable $P$-module
(which is equivalent to saying $W$ is a cyclic $FP$-module).

In this case we show that $\dim H^j(P,W)=1$ unless $W$ is free
(in which case the dimension is $0$) by induction on $j$.  If
$j=0$, this is clear.  So assume that $W$ is not free.
Since $W$ is self cyclic and self dual, it embeds in a rank
one free module $V$. 
Then $H^i(P,V)=0$ and by Lemma \ref{cohomology sequence}, $H^i(P,W)
\cong H^{i-1}(P,V/W)$ and so is $1$-dimensional (since $V/W$
is nonzero and cyclic).
\end{proof}

The next result is standard -- cf. \cite[p. 91]{relmod}.

\begin{lemma} \label{restriction}  
If $H$ contains a Sylow $p$-subgroup of $G$,
then the restriction map $H^i(G,M) \rightarrow H^i(H,M)$
is an injection.
\end{lemma}

The next result
 is an easy consequence of the Hochschild-Serre spectral sequence
\cite[p. 337]{maclane}. See also \cite{holt1}.

\begin{lemma} \label{spectral}
Let $N$ be a normal subgroup
of $G$, $F$ a field and  $M$  an $FG$-module. Then
$\dim H^q(G,M) \le \sum_{i+j=q} \dim H^i(G/N,H^j(N,M))$.
\end{lemma}

We single out the previous lemma for the cases $q = 1,2$.
See \cite[pp. 354--355]{maclane} or \cite[Lemma 2.1]{holt}.

\begin{lemma} \label{usual} Let $N$ be a normal subgroup
of $H$ and let $M$ be an $FH$-module.  Then
$$
\begin{array}{llcl}
(1) &\dim H^1(H,M) &\le &\dim H^1(H/N,M^N) + \dim H^1(N,M)^H, \ and \\ 
(2) &\dim H^2(H,M) &\le& \dim H^2(H/N,M^N) + \dim H^2(N,M)^H + \\
 &&& \dim H^1(H/N,H^1(N,M)).
\end{array}
$$
\end{lemma}

We shall use the following well known statements without comment.

\begin{lemma} \label{perfect}  If $G$ is perfect, then $H^1(G,\F_p)=0$ 
and $\dim H^2(G,\F_p)$ is the $p$-rank of
the Schur multiplier of $G$.
\end{lemma}

We also use the following consequence of the K\"unneth formula.

\begin{lemma}  \label{kunneth}  Let $F$ be  a field 
and let $H=H_1 \times \cdots \times H_t$ with
the $H_i$ finite groups.  Let $M_i$ be an irreducible $FH_i$-module for each $i$
and set $M= \otimes_{i=1}^t M_i$,
an irreducible $FH$-module.  Then
\begin{enumerate}
\item  $ H^r(H,M) = \oplus_{(e_i)} H^{e_1}(H_1,M_1) \otimes \cdots \otimes 
                                  H^{e_t}(H_t,M_t)$, 
where the sum is over all $(e_i)$ with with the $e_i$ non-negative
integers and $\sum e_i=r$.
\item If $H_i$ acts nontrivially on $M_i$ for each $i$, 
then $H^r(H,M)=0$ for
$r < t$ and  $\dim H^t(H,M) = \prod \dim H^1(H_i,M_i)$.
\item   If each $H_i$ is quasisimple and each $M_i$ is nontrivial, 
then $\dim H^t(H,M) \le \dim M/2^t$.
\item If the $H_j$ are perfect for $j >1$, $M_1$ is nontrivial and $M_j$
is trivial for $j >1$, then $H^2(H,M) \cong H^2(H_1,M_1)$.
\end{enumerate}
\end{lemma}

\begin{proof}  The first statement is just the K\"unneth formula as given
in \cite[3.5.6]{benson}, and the second statement follows immediately since
$H^0(H_i,M_i)=0$.   If $H_i$ is quasisimple, then (2) and   (1.2)
imply (3).  Finally (4) follows from (1) and the fact that, by Lemma \ref{perfect},
$H^1(H_j,M_j)=0$ for $j > 1$.
\end{proof}
 
Note that there are quite a number of terms involved in $H^r(H,M)$
in the lemma above.  Fortunately, when $r$ is relatively small, most
terms will be $0$.

See \cite[35.6]{baba} for the next
result.

\begin{lemma} \label{inflres2}  Assume that $N$ is normal in
$H$ and $H^{r-1}(N,M)=0$.  Then there is an exact sequence
$$
0 \rightarrow H^r(H/N,M^N) \rightarrow H^r(H,M) \rightarrow H^r(N,M)^H.
$$
\end{lemma}

We single out a special case of the previous lemma.

\begin{corollary} \label{coprime} Let $H$ be a finite group
with a normal subgroup $N$.  Let $M$ be an $FH$-module.  
\begin{enumerate}
\item If $M^N= H^{j-1}(N,M)=0$, then the restriction map  $H^j(H,M) \rightarrow
H^j(N,M)$ is injective. 
\item If $N$ has order that is not a multiple of the
 characteristic of $M$ and $M^N=0$, then $H^j(H,M)=0$ for all $j$.
\item If $N$ has order that is not a multiple of the
 characteristic of $M$ and $M^N=M$, then $H^j(H/N,M) \cong H^j(H,M)$ for all $j$.
\end{enumerate}
\end{corollary}

\begin{proof}   (1) is an immediate consequence of the previous lemma.
Under the assumptions of (2), $M$ is a projective $FN$-module and so
$H^j(N,M)=0$ for all $j > 0$ and $H^0(N,M)=0$ by hypothesis.  Thus (2)
follows by induction on $j$ and (1).  Note that (3) is a special
case of Lemma \ref{inflres2}.   
\end{proof}
 
The previous corollary in particular illustrates the well known result
that all higher cohomology groups for finite groups vanish in characteristic
$0$.  So we will always assume our fields have positive characteristic
in what follows.

It is also convenient to mention a special case of Lemma \ref{inflres2}  for $H^1$.

\begin{lemma} \label{h1 nonfaithful}   Let $G$ be a finite group with
$p$ a prime.  Let $N$ be a normal $p$-subgroup of $G$ and $V$ an $\F_pG$-module
with $N$ acting trivially on $V$.  Then
$\dim H^1(G,V) \le \dim H^1(G/N,V) + \dim \hom_G (N,V)$.
\end{lemma}

\begin{lemma} \label{direct product relations}  Let $A$ and $B$
be quasisimple groups with trivial Schur multipliers, and let
$G= A \times B$.
Then $d(G)=2$ and  $\hat{r}(G) = \max\{\hat{r}(A), \hat{r}(B)\}$.
\end{lemma}

\begin{proof}  Since $d(A)=d(B)=2$, it follows  that $d(G) = 2$  unless 
possibly $A \cong B$.  In that case $d(G)=2$ follows from the fact
that the set of generating
pairs of a finite simple group are not a single orbit under the
automorphism group (e.g., use the main result of \cite{gurkan}).
The last statement now  follows by Lemma \ref{kunneth} and (1.1).
\end{proof}  

Since (1.1) does not give an explicit presentation, we cannot give
one in the previous result.   It would be interesting to do so.

The next result is an interesting way of giving profinite presentations
with fewer relations than one might expect by giving presentations
with more generators than the minimum required. Recall that a profinite
presentation for a finite group $G$ is a 
 free profinite group $F$ and a finite subset $U$ of $F$ 
  such that if $R$ is the closed normal subgroup
generated by  $U$, $G \cong F/R$.

We show in the next result that
if $G$ has a profinite presentation with $d(G) + c$ generations
and $e$ relations, then it has a profinite presentation with
$d(G)$ generators and $e - c$ relations.  Often, we will give
profinite presentations with more than the minimum number of generators
required and so we deduce the existence of another profinite
presentation with $d(G)$ generators and fewer relations. We do not
know how to make this explict and if this is true for discrete presentations.
  Indeed, the best result we know
is that if $G$ has a (discrete) presentation with $r$ relations,
then it is has a (discrete) presentation with $d(G)$ generators and
$r + d(G)$ relations
(see \cite[Lemma 2.1]{pap1}).  

If $M$ is an $\F_pG$-module,
let $d_G(M)$ be the minimum size of a generating set for
$M$ as an $\F_pG$-module.  The key result is in
\cite[Theorem 0.2]{lub2}, which asserts that if $G=F/R$ is a finite
group, $F$ is a free profinite group and
$R$ is a closed normal subgroup of $F$, then the minimal
number of elements needed to generate $R$ as a closed normal 
subgroup of $F$ is equal to $\max_p \{d_G(M(p)\}$, where
 $M(p)$ is the $G$-module $R/[R,R]R^p$ and $p$ ranges over all
 primes.  Moreover, by
\cite[2.4]{relmod} the structure of $M(p)$
depends only on the rank of $F$.

\begin{lemma}  \label{saving} Let $G$ be a finite 
group.  
Consider a profinite presentation $G=F/R$ where $F$
is the free profinite group on $d(G) + c$ generators.
Let $e$ to be the minimal number of elements required 
to generate  $R$ as a closed
normal subgroup of $F$.  Then $\hat r(G) = e - c$. 
 In particular, 
the minimal number of relations
occurs  when the   number of generators is minimal, and only in that
case.  
\end{lemma}

\begin{proof}  
 Set $M=R/[R,R]$ and $M(p)=M/pM$ for $p$ a prime.  So $M$
is the relation module for $G$ in this presentation
and $M(p)$ is an $\F_pG$-module.

As noted above,  $R$ is normally generated (as a closed subgroup)
in $F$ by $e$ elements, where $e = \max_p \{d_G(M(p)\}$.
 Also as noted above the structure of $M(p)$ only depends on
the number of generators for $F$ and not on the particular
presentation.  So we may assume that all but $d(G)$ generators
in the presentation are sent to $1$, whence we see that
$M(p)= N(p) \oplus X_p$ where $X_p$ is a free $\F_pG$-module
of rank $c$ and $N(p)$ is the $p$-quotient of the relation
module for a minimal presentation.   

Now the first statement follows from  the elementary fact that,
if  an $\F_pG$-module $Y$ can be generated
by $s$ elements   but no fewer, then the   $\F_pG$-module
  $Y \oplus \F_pG$ is generated by   $s+1$ elements  but  no fewer. 
Indeed, this holds for any finite dimensional algebra $A$ over a field
-- for by Nakayama's Lemma, we may assume that $A$ is semisimple
and so reduce to the case that $A$ is a simple algebra, where the
result is clear.

The last statement is now an immediate consequence.
\end{proof}

\begin{lemma}  \label{H2 for abelian} Let $G$ be a finite group with a normal  
abelian $p$-subgroup $L$.    Let $V$ be an irreducible $\F_pG$-module. 
\begin{enumerate}
\item There is an exact sequence of $G$-modules,
$$
0 \rightarrow \mathrm{Ext}_{\mathbb{Z}}(L,V) \rightarrow H^2(L,V) \rightarrow 
\wedge^2( L^*) \otimes V \rightarrow 0.
$$
\item 
$\dim H^2(L,V)^G \le \dim ((L/pL)^* \otimes V)^G 
+ \dim_F (\wedge^2 (L/pL)^* \otimes V)^G$.
\item If $G=L$, then $\dim H^2(G,\F_p) =  d(d+1)/2$ where $d=d(G)$.
\end{enumerate}
\end{lemma}

\begin{proof}  Since $G$ acts irreducibly on $V$, it follows that $L$
acts trivially on $V$.

It is well known (cf. \cite[p. 127]{brown} or \cite{bell}) that when $L$ is abelian
and acts trivially on $V$, 
there is a (split) short exact sequence as in (1) in the category
of abelian groups. Here $\mathrm{Ext}_{\mathbb{Z}} (L,V)$ is the subspace 
of $H^2(L,V)$ corresponding to
abelian extensions of $L$ by $V$.
 The natural maps are  $G$-equivariant,  giving (1).
   Note that $\mathrm{Ext}_{\mathbb{Z}}(L,V)
\cong \mathrm{Hom} (L/pL,V) \cong (L/pL)^* \otimes V$ even as $G$-modules.
Also,  $ \wedge^2(L^*) \otimes V \cong \wedge^2 (L/pL)^* \otimes V$
since $V$ is elementary abelian.

Taking $G$-fixed points gives (2), and taking $G=L$ and $V=\F_p$ gives (3).
\end{proof}

\begin{lemma} \label{multiplicity free}
Let $T$ be a finite cyclic group of order $(q-1)/d$
acting faithfully on the irreducible $\F_pT$-module
$X$ of order $q=p^e$.  
Set $Y=\wedge^2(X)$.
Assume either that $d < p$ or that both $d=3$ and $q > 4$.
Then
\begin{enumerate}
\item   $Y$ is multiplicity free as a $T$-module; and
\item   $X$ is not isomorphic to a submodule of $Y$.
\end{enumerate}
\end{lemma}

\begin{proof}   Let $x \in T$ be a generator.  Thus,
$x$ acts 
on $V$ with eigenvalue $\lambda \in \F_q$ of order
$(q-1)/d$.    It is straightforward to see that
$Y \otimes_{\F_p} \F_q$ is a direct sum of submodules
on which $x$ acts via $\lambda^{(p^i + p^j)}$ where
$1 \le i < j < e$.   These submodules are all nonisomorphic
(if not, then $d(p^i+p^j) \equiv d(p^{i'} + p^{j'})$ modulo $p^e-1$
for some distinct pairs $\{i,j\}$ and $\{i', j' \}$)
and similarly are not isomorphic to $X$.
\end{proof}

\begin{lemma} \label{exterior square}  Let $G$ be a finite group.
Let $V$ be an irreducible $\F_pG$-module of dimension $e$.
Then $\wedge^2(V)$ can be generated by $e-1$ elements as an
$\F_pG$-module.  In particular, $\dim \mathrm{Hom}_G(\wedge^2(V),W)
\le (e-1) \dim W$ for any $\F_pG$-module $W$.
\end{lemma}

\begin{proof}   Choose a basis $v=v_1, \ldots, v_e$ for $V$.
It is clear that $v_1 \wedge v_j, 2 \le j \le e$, is a generating
set for $\wedge^2(V)$ as a $G$-module, which proves the first
statement.  The second statement is a trivial consequence of
the first.
\end{proof}

We will use the following elementary result to bound
the number of trivial composition factors in a module.

\begin{lemma} \label{trivial cf}
Let $G$ be a finite group and $F$ a field of characteristic $p$.
Let $M$ be an $FG$-module and let $J$ be a subgroup of $G$.
\begin{enumerate}
\item If $M^G=0$ and 
 $G$ can be generated by $2$ conjugates of $J$, then
$\dim M^J  \le (1/2) \dim M$.
\item If $|J|$ is a not a multiple of $p$, then the number
of trivial $FG$ composition factors is at most $\dim M^J$.
\end{enumerate}
\end{lemma}

\begin{proof}   If $G= \langle J,K \rangle$ for some conjugate
$K$ of $J$, then
$M^J \cap M^K=M^G= 0$, whence (1) holds.  In (2),  since $J$ has order
coprime to the characteristic of $F$, $M=M^J \oplus V$
where $J$ has no trivial composition factors on $V$.
Thus, the number of $J$-trivial composition factors
is at most $\dim M^J$ and so this is also an upper bound
for the number of $G$-trivial composition factors.
\end{proof}

\section{Covering Groups}\label{covering}

We will also switch between the simple group and a covering group.
Recall that a group $G$ is quasisimple if it is perfect and $G/Z(G)$
is a nonabelian simple group.  Recall also the definition of $h'(G)$ from (1.3).
 
If $N$ is a normal of $G$ and $M$ is a $G$-module with
$M^N=M$, then we may and do view $M$ as a $G/N$-module.

\begin{lemma} \label{covering lemma}  Let $G$ be a finite quasisimple group.
Let $r$ be prime and let   $Z$ be a central $r$-subgroup of $G$.
Let $M$ be a nontrivial irreducible $FG$-module with $F$ a field of 
characteristic $p$.   
\begin{enumerate}
\item If $M^Z = M$, then $H^1(G/Z,M) \cong H^1(G,M)$.
\item If $r \ne p$, then either $Z$ acts nontrivially and $H^2(G,M)=0$, or 
$Z$ acts trivially
and $H^2(G/Z,M) \cong H^2(G,M)$.
\item If $r=p$, then $Z$ acts trivially on $M$, and  
$$
\dim H^2(G/Z,M) \le
\dim H^2(G,M) \le \dim H^2(G/Z,M) + c \dim H^1(G/Z,M),
$$
where $c$ is the rank of $Z$.  In particular, 
$$ 
\dim H^2(G,M) \le \dim H^2(G/Z,M) + \dim M
$$
\item $h'(G/Z) \le h'(G) \le h'(G/Z) + 1$.
\end{enumerate}
\end{lemma}

\begin{proof} 
The first statement follows by Lemma \ref{usual}.  (2) is included
in Lemma \ref{coprime}.

So assume that $r = p$. 

 We use the inequality from Lemma \ref{usual}(2):
$$
\dim H^2(G,M) \le H^2(G/Z,M^Z) + \dim H^2(Z,M)^G + \dim H^1(G/Z, H^1(Z,M)).
$$

By Lemma \ref{H2 for abelian},
$\dim H^2(Z,M)^G \le \dim \mathrm{Hom}_G(Z,M) + \dim \mathrm{Hom}_G(\wedge^2(Z),M)=0$
since $M^G=0$.  So the
middle term of the right hand side above is $0$.  
Now $H^1(Z,M) \cong \mathrm{Hom}(Z/pZ,M)$.  Since $Z/pZ$ is a direct sum of $c$ copies
of the trivial $\F_pG$-module, where $c$ is at most the rank of $Z$,
   $\mathrm{Hom}(Z/pZ,M)$ is isomorphic to $c$
copies of $M$
(as a $G$-module).   Thus, $\dim H^1(G/Z, H^1(Z,M)) \le c \dim H^1(G/Z,M)$
 and so the second inequality in (3) holds.  Since $c \le 2$ \cite[pp. 313--314]{gls3},
and $\dim H^1(G/Z,M) \le (\dim M)/2$, the last part of (3) follows.

Finally we show that $\dim H^2(G/Z, M) \le \dim H^2(G,M)$.
 We use relation modules for
this purpose.   Write $G=F/R$ where $F$ is free of rank $d(G)$.  Let $S/R$
be the central subgroup of $F/R$ corresponding to $Z$.
Let $R(p)=R/[R,R]R^p$ be the $p$-relation module
for $G$ and $S(p)=S/[S,S]S^p$ the $p$-relation module for $G/Z$.   Clearly, 
there
is a $G$-map $ \gamma: R(p) \rightarrow S(p)$ with $S(p)/\gamma(R(p))$ having trivial $G$-action.
Thus, the multiplicity of an irreducible nontrivial $G$-module $M$ 
in $S(p)/\mathrm{Rad}(S(p))$
is at most the multiplicity of $M$ in $R(p)/\mathrm{Rad}(R(p))$.   
Since these multiplicities
are $\dim H^2(G,M) - \dim H^1(G,M)$ and 
$\dim H^2(G/Z,M)-\dim H^1(G/Z,M)$, and since,  by (1),
$\dim H^1(G/Z,M)=\dim H^1(G,M)$, the inequality follows.

Now (4) follows  from (1), (2), (3) and (1.2).
\end{proof}

We can interpret this for profinite presentations.  Recall
that $\hat{r}(G)$ is the minimal number of relations required
among all profinite presentations of the finite group $G$.

\begin{corollary} \label{covering relations}
Let $G$ be a quasisimple group with a central subgroup $Z$.  
\begin{enumerate}
\item  $\hat{r}(G/Z) \le \max \{ \hat{r}(G), 2 + \mathrm{rank}(J)\}
\le \max\{ \hat{r}(G), 4 \}$,
where $J = H^2(G/Z, \mathbb{C}^*)$ is the Schur multiplier of $G/Z$.
\item  $\hat{r}(G) \le  \hat{r}(G/Z) + 1$.
\end{enumerate}
\end{corollary}

\begin{proof}  We first prove (1).
Let $M$ be an irreducible $G/Z$-module.  We may view $M$
as a $G$-module.  First suppose that $M$ is trivial.
Then $\dim H^2(G/Z,M) \le \mathrm{rank}(J) $. 
 
 Now assume that $M$ is nontrivial. Then by Lemma \ref{covering lemma}(3), 
 $$
 \dim H^2(G,M) - \dim H^1(G,M)  \ge \dim H^2(G/Z,M) - \dim H^1(G/Z,M).
 $$
It follows by (1.1) that either $\hat{r}(G/Z) = 2 + \mathrm{rank}(J) \le 4$
or $\hat{r}(G/Z) \le \hat{r}(G)$, whence the result holds.

We now prove (2). Note that $d(G) = d(G/Z)$.
  Let $M$ be an irreducible $\F_pG$-module which
achieves the maximum $\hat{r}(G)$ in (1.1).  If
$M$ is trivial, then $\dim H^2(G/Z,M) \ge \dim H^2(G,M)$
and so $\hat{r}(G) \le \hat{r}(G/Z)$ in this case.

Suppose that $M$ is nontrivial.  If $Z$ acts nontrivially on 
$M$, then $H^j(G,M)=0$ for all $j$, a contradiction.
So we may assume that $Z$ is trivial on $M$. Then  by Lemma \ref{covering lemma}(4),
$$
\frac{\dim H^2(G,M)}{\dim M} \le \frac{\dim H^2(G,M)}{\dim M} + 1,
$$ 
As noted in the previous proof, $\dim H^1(G,M)=\dim H^1(G/Z,M)$.
 Now apply (1.1).
\end{proof}

The previous two results allow us to work with covering groups rather than 
simple groups.   So if we prove that the universal central extension $G$
of a simple group $S$ can be presented profinitely with $r \ge 4$ relations, the
same is true for any quotient of $G$ (and in particular for $S$).  Conversely,
if a finite simple group $S$ can be presented with $r$ profinite relations, then any quasisimple
group with central quotient $S$ can be presented with $r+1$ profinite relations.

\section{Faithful Irreducible Modules and Theorem C} \label{faithful section}

In this section we show that a bound for  $\dim H^2(G,M)/ \dim M$ with $G$ simple 
and $M$ a nontrivial irreducible $\F_pG$-module implies
a related bound for arbitrary finite groups and
irreducible faithful modules.   In particular, this shows how Theorem B
implies Theorem C. 

 It is much easier to prove
that $\dim H^1(G,M) \le \dim H^1(L,M)$ if $M$ is an irreducible faithful
$FG$-module and $L$ is any component of $G$.  See \cite{gur} and 
    Lemma \ref{faithful lemma} (5) below for a stronger result.
   
   For $H^2$, the reduction to simple groups is 
    more involved,  and it is not clear that the constant one
obtains for simple groups is the same   constant 
for irreducible faithful modules.  
   Holt \cite{holt} used a similar reduction   for a   weaker
   result, but it is not sufficient to appeal to his results.
 
If $L$ is a nonabelian
simple group, let  
$$h_i(L) = \max\{ \dim H^i(L,M)/\dim M\}, $$
where  the maximum is taken over all nontrivial irreducible $\F_pL$-modules
and all $p$.
So $h_2(L)=h'(L)$ as defined in (1.3). 
Let $o_p(L)$ denote the maximal dimension of any section of $\mathrm{Out}(L)$ that
is an elementary abelian
$p$-group (this is called the sectional $p$-rank of $\mathrm{Out}(L)$).
Let $o(L) = \max_p \{o_p(L)\}$.  
We record some well known facts about this.  See \cite[Chapter 4]{gls3}.

\begin{lemma} \label{out}  Let $L$ be a nonabelian finite simple group.
Then $o_p(L) \le 2$ for $p$ odd,  and $o(L) \le 3$.   
\begin{enumerate}
\item   If $L = A_n, n \ne 6$ or $L$ is sporadic, then
$\mathrm{Out}(L)$  has order at most $2$,  and $o(L) \le 1$.  
\item  $\mathrm{Out}(A_6)$ is elementary abelian of order $4$.
\item  Assume that $L$ is of Lie type.    Then  $o_2(L) \le 2$ unless    
$L \cong \PSL(d,q)$ with $q$ odd and  $d > 2$ even, 
or $L \cong P\Omega^+(4m,q)$ with $q$ odd.  
\end{enumerate}
\end{lemma}

\begin{lemma} \label{faithful lemma}  Let $F$ be a field of characteristic $p$, 
$G$ a finite
group and $M$ an irreducible $FG$-module that is faithful for $G$.
Assume that $H^k(G,M) \ne 0$ for some $k > 0$
(and so in particular, $p > 0$).
\begin{enumerate}  
\item  $O_p(G)=O_{p'}(G)=1$;  in particular $G$ is not solvable.
\item  Let $N=\FF^*(G)$.  For some $t \ge  1$, 
$N$ is a direct product of $t$ nonabelian simple groups.
\item  $G$ has at most $k$ minimal normal subgroups.
\item  If $W$ is an irreducible $FN$-submodule of $M$,  and  if two distinct
components of $G$ act nontrivially on $W$, then $H^1(G,M)=0$ and
 $\dim H^2(G,M) \le (\dim M)/4$.  In particular, this is the case
 if $G$ does not have a unique minimal normal subgroup. 
\item Suppose that $N$ is the unique minimal normal subgroup of $G$
and $L$ is a component of $G$.
\begin{enumerate} 
\item   $\dim H^1(G,M) \le \dim H^1(L,W)$ for $W$  any  
irreducible $L$-submodule of $M$ with $W^L=0$.
\item  $\dim H^2(G,M) \le (h_1(L)(o(L) +1/2) + h_2(L)/t) \dim M$.
\item If  $N=L$, then $\dim H^2(G,M) \le (h_2(L)+1) \dim M$.
 \end{enumerate}
 \end{enumerate}
\end{lemma}

\begin{proof}  The hypotheses imply that $M$ is not projective
and so $p > 0$.  Since $M$ is faithful and irreducible, $O_p(G)=1$.
If $O_{p'}(G) \ne 1$, then by 
Corollary \ref{coprime}, $H^d(G,M)=0$ for all $d$.
So (1) and (2) hold.    

We may assume that $F$ is algebraically closed
(see Lemma \ref{extension of scalars}).
Write $N=N_1 \times \cdots \times N_e$ where the $N_i$
are the minimal normal subgroups of $G$.  Let $W$ be an irreducible
$FN$-submodule.  Then $W=W_1 \otimes \cdots \otimes W_e$ 
is a tensor product of irreducible $FN_i$-modules. 
 Since $M$ is faithful and irreducible, $M^{N_i} = 0$.
 In particular, each $N_i$ is faithful on $W$, whence $W_i$ is nontrivial
 for each $i$.

It follows by Lemma \ref{kunneth}, that $H^j(N,W)=0$ for $j < e$.   
So $H^j(N,gW)=0$ for any $g \in G$ with $j < e$. 
Since $M$ is a direct sum of irreducible $N$-modules of the form $gW$,
$g \in G$, $H^j(N,M)=0$ for $j < e$.
It follows by   Lemma \ref{inflres2} that $H^j(G,M)$ embeds in
$H^j(N,M)=0$ for $j < e$, whence (3) follows (see also \cite{willems}).

The same argument shows that $H^j(G,M)=0$ if there is an irreducible
$FN$-submodule $W$ of $G$ in which at least $j+1$ components act
nontrivially.  If precisely $j$ components act nontrivially, the
argument shows that $\dim H^j(G,M) \le (\dim M)/2^j$.   
Since $N_i$ has no fixed
points on $M$, it follows that at least $e$ components act nontrivially on
any irreducible $FN$-submodule, whence (4) holds.

So assume that $N$ is the unique minimal normal subgroup of $G$.
Write $N=L_1 \times \cdots \times L_t$
with the $L_i$ isomorphic nonabelian simple groups.   
Set $L=L_1$ and $h_1=h_1(L)$.

Let $W$ be an irreducible $FN$-submodule with $W^L =0 $.   Suppose first that
$L_j$ acts nontrivially on $W$ for some $j > 1$.  Arguing as above and using 
Lemma \ref{kunneth} and Lemma \ref{usual}   shows that $H^1(G,M)=H^1(N,M)=0$.
Similarly, we see that  $\dim H^2(N,W) \le h_1^2 \dim W$ and
 $\dim H^2(G,M) \le \dim H^2(N,M) \le h_1^2 \dim M$. Using (1.2) shows
 that (5b) follows in this case.  
  
   So to complete the proof of all parts of (5), 
we may assume that $L_j$ is trivial on $W$ for $j > 1$ (for case (5c),
there is only one component).
 
We now prove the first part of (5).
Let $U$ be the largest $L$-homogeneous submodule of $M$ containing $W$
(i.e.  $U$ is the $L$-submodule generated by the $L$-submodules isomorphic
to $W$).

Let $I$ be the stabilizer of $U$ in $G$.  Note that $I \le N_G(L)$
(since for $j \ne 1$, $L_j$ acts trivially on $U$).  
Since $M$ is irreducible, $U$ is an irreducible $I$-module.
Let $R=LC_I(L)$.
By Lemma \ref{shapiro},    $H^k(G,M) \cong H^k(I,U)$.
By Lemma \ref{inflres2},
$\dim H^1(I,U) \le \dim H^1(R,U)$. Since $R = L \times C_I(L)$, 
$U$ is a direct sum of modules of
the form $W \otimes X$ where 
each $X$ is an irreducible $C_I(L)$-module.   Since $W$ is irreducible,
it follows that either all $X$ are trivial $C_I(L)$-modules or none
are.  In the latter case, $H^1(R,W)=0$ by Lemma \ref{kunneth}, and
so $H^1(G,M)=0$.  So $C_I(L)$ acts trivially on $U$.  Set $J=I/C_I(L)$.  

By Lemma \ref{usual},  $\dim H^1(I,U) \le \dim H^1(J,U) + \dim H^1(C_I(L),U)^I$.
Since $C_I(L)$ is trivial on $W$, $H^1(C_I(L),W)^I$ is the set of $I$-
homomorphisms from $C_I(L)$ to $U$.  Since $L$ acts trivially on $C_I(L)$ and 
$U^L=0$, $H^1(C_I(L),U)^I=0$. 
Thus $ \dim H^2(I,U) \le \dim  H^2(J,U)$.
Note that $J$ is almost simple with socle $L$ and that $J$ acts irreducibly
on $U$.   So we have reduced the problem to the case $t=1$, $C_G(L)=1$ and $G=I$.
Now use the fact that $G/L$ is solvable (which depends on the classification
of finite simple groups) and let $D/L$ be a maximal
normal subgroup of $I/L$.   So $I/D$ is cyclic of prime order $s$.
If $D$ does not act homogeneously, then $U$ is induced and we can
apply Lemma \ref{shapiro}.  So we may assume that $D$ acts homogeneously.
It follows by Clifford theory and the fact that $I/D$ is cyclic that
$D$ acts irreducibly on $U$.  By Lemma \ref{inflres2},
$\dim H^1(I,U) \le \dim H^1(D,U)^I \le \dim H^1(D,U)$, and so
by induction (on $|I:L|$), $\dim H^1(D,U) \le \dim H^1(L,W)$,  as required.

Now consider $H^2(G,M)$ in (5).   Let $M_i=[L_i,M]$.  So $M$ is the direct
sum of the $M_i$.
Another application of Lemma \ref{usual},
together with the fact that $M^N=0$,  shows that
$$
\dim H^2(G,M) \le \dim H^2(N,M)^G + \dim H^1(G/N,H^1(N,M)).
$$
Now $H^2(N,M)$ is the direct sum of the $H^2(N,M_i) \cong H^2(L_i,M_i)$ 
(by Lemma \ref{kunneth}), and
$G$ permutes these terms transitively.  Thus $\dim H^2(N,M)^G \le \dim H^2(L,M_1) \le 
(h_2(L)/t) \dim M$.   

Similarly, $H^1(N,M)$ is the direct sum of the $H^1(L_i,M_i)$, 
and $G/N$ permutes these.  Thus, $H^1(N,M)$ is an induced 
$G/N$-module,  and so by Shapiro's Lemma (Lemma \ref{shapiro})
$H^1(G/N,H^1(N,M)) \cong H^1(N_G(L)/N,H^1(L, M_1))$.  Note that $M_1$ is an 
irreducible
$N_G(L)$-module and is a faithful $L$-module (since $M$ is irreducible for $G$).  

Let $P$ be a Sylow $p$-subgroup of $N_G(L)$.  Let $K = \cap_i N_P(L_i)$
and note that $K$ is normal in $P$.  
   Then $KN/N$ can be generated by $o(L)t$ elements
(by induction on $t$).    By \cite[Theorem 2.3]{ag}, $P/K$ can be generated
by  $\lfloor t/2 \rfloor$ elements, whence $PN/N$ can be generated by
at most $(o(L) + 1/2)t$ elements.  
Since $\dim H^1(L,M_1) \le h_1(L)(\dim M_1)$, it follows
that 
$$
\begin{array}{cl}
\dim H^1(N_G(L)/N,H^1(L, M_1)) & \le  \dim H^1(PN/N, H^1(L,M_1)) \\
                               & \le  h_1(L)(o(L) + 1/2)t  \dim M_1 \\ 
                               & = h_1(L)(o(L) +1/2)\dim M.
\end{array}
$$
 Thus,
$$
\frac{\dim H^2(G,M)}{\dim M} \le h_1(L)(o(L) +1/2) + h_2(L)/t.
$$
This gives (5b).

We now prove (5c).  So assume that $t=1$.
Then 
 $PN/N$ can be generated by $o(L)$ elements and so we get the bound
 $\dim H^2(G,M) / \dim M \le h_2(L) +o(L)h_1(L)$.

    By  (1.2), $h_2(L) +o(L)h_1(L) \le h_2(L)+1$ unless $o(L) > 2$.  
    However, we have already
 noted that  in the cases where  $o(L) = 3$,  $L$ must be a group
 of Lie type over a field of odd characteristic and the $p$-subgroup
 of $\mathrm{Out}(L)$ requiring $3$ generators must be a $2$-subgroup.  It follows
 by \cite{hoffman} that in all these cases $\dim H^1(L,M) \le (\dim M)/3$,
 whence the result holds.
\end{proof}

 \begin{theorem} \label{faithful}  Let  $F$ be a field, 
 $G$ be a finite group with $M$ a faithful
 irreducible $FG$-module.  If $H^2(G,M) \ne 0$, then $G$ has a component $L$
 and 
 $$\frac{\dim H^2(G,M)}{ \dim M}  \le \max \{7/4, h_2(L) + 1 \}.$$
 \end{theorem}

 \begin{proof}  
    Since
 $H^2(G,M) \ne 0$, the previous lemma applies. If $G$ has more than
 one minimal normal subgroup, then $\dim H^2(G,M) \le (\dim M)/4$
 by Lemma \ref{faithful lemma}(4),  and
 the result holds.  So we may assume that $G$ has a unique minimal
 normal subgroup $N$, that $L$ is a component of $G$ and that $N$ is a direct
 product of $t$ conjugates of $L$.
 Now the bound in (5b) of the previous lemma
 applies.

 As we have noted above,  $o(L) \le 3$ with equality implying that $G$ is
 a finite group of Lie type  $A$ of rank at least $3$ or of type
 $D$ of rank at least $4$.  If $o(L) \le 2$, it follows from (1.2) 
 that $h_1(L)(o(L) +1/2) \le 7/4$.  If $o(L)=3$, 
  it follows from \cite{gurhoff} that $h_1(L) \le 1/3$, 
 whence $h_1(L)(o(L) +1/2) \le 5/4$.  
 
 Let $t$ be the number of components of $G$.  If  $t=1$, the result
 follows by (5c) of the previous lemma.
  So assume that $t > 1$.
 By (5b) of the previous result,  $\dim H^2(G,M) / \dim M 
 \le h_1(L)(o(L) +1/2) + h_2(L)/2$. The right hand side
 is bounded above by $(5/4) + h_2(L)/2 \le \max\{7/4,h_2(L)+1\}$,
 and the result follows. \end{proof}
  
  An immediate consequence of the previous result is:
  
  \begin{corollary}  Theorem B implies Theorem C. 
  \end{corollary}

\section{Alternating and Symmetric Groups} \label{alternating}

We will need the following better bound
 for $H^1$ for alternating groups given in  \cite{gurkim}.
 
\begin{theorem} \label{H1alt} 
Let $G=A_n, n > 4$.  If $F$ is a field
and $M$ is an irreducible $FG$-module, then   
\begin{enumerate}
\item 
$\dim H^1(G,M) \le (\dim M)/(f-1)$ where $f$ is the largest
prime such that $f \le n - 2$;
\item $\dim H^1(G,M) \le (2/n) \dim M$ for $n > 8$; 
\item $\dim H^1(A_8, M) \le (\dim M)/6$; and 
\item if $F$ has characteristic $p$, then 
$\dim H^1(G,M) \le (\dim M)/(p-2)$.
\end{enumerate}
\end{theorem}

The goal of this section is to prove the following results:

\begin{theorem} \label{H2alt}
Let $G=A_n$ or $S_n, n > 4$, and let $p$ be a prime.  Let $F$
be a field of characteristic $p$ and $M$ an $FG$-module.
\begin{enumerate}
\item If $p > 3$, then  
$\dim H^2(G,M) \le (\dim M)/(p-2) \le  (\dim M)/3$.
\item If $p=3$, then $\dim H^2(G,M) \le \dim M$ with equality
if and only if $n=6$ or $7$ and $M$ is the trivial module.
\item If $p=2$, then $\dim H^2(A_n,M) \le (35/12) \dim M$.
\item If $p=2$, then $\dim H^2(S_n, M) < 3\dim M$.
\end{enumerate}
\end{theorem}

These results are likely quite far from best possible.  
By (1.1) and Corollary \ref{covering relations}, this   gives:

\begin{corollary} \label {an relations}
\begin{enumerate}
\item  $\hat{r}(A_n) \le 4$,
\item $\hat{r}(S_n) \le 4$, and
\item if $G$ is any quasisimple group with $G/Z(G) = A_n$, then
$\hat{r}(G) \le 5$.
\end{enumerate}
\end{corollary}

Almost certainly, it is the case that $\hat{r}(A_n)=\hat{r}(S_n)=3$
for $n > 4$. 
Since the Schur multipliers of $A_n$ and $S_n$ are nontrivial 
for $n > 4$, $\hat{r}(A_n)$ and $\hat{r}(S_n)$ are both 
at least $3$.  
The proof we give says very little about finding specific relations.
It would be quite interesting to pursue this further.

The main idea is to pass to a subgroup containing a Sylow $p$-subgroup of $G$
and having a normal subgroup that is a direct product of
alternating groups.  We then use induction together with the results
of Section \ref{prelim1}.

We do this first for $p > 3$, then for $p=3$ and finally for $p=2$.
If $p > 3$, each of these smaller alternating groups  is simple and
has Schur multiplier prime to $p$.  If $p=3$, there may be an $A_3$
factor.  Also,    $A_6$ and $A_7$ have Schur multipliers of order $6$.
For $p=2$, there may be solvable factors, all Schur multipliers
have even order and there are  further complications as well.

\subsection{{$p > 3$}} \hfil\break

For this subsection, let $F$ be an algebraically closed field
of characteristic $p > 3$.   Our goal is to prove
the following result,  which
includes  Theorem \ref{H2alt} for $p > 3$.

\begin{theorem} \label{alt p > 3}
Let $p > 3$ be a prime.  Let $G=A_n$ and $F$ a field of characteristic
$p$.   If $M$ is an $FG$-module, then $\dim H^j(G,M) \le (\dim M)/(p-2)$
for $j = 1,2$.
\end{theorem}

\begin{proof}   
We induct on $n$.   If $p < n$, all $FG$-modules are projective
and so $H^j(G,M)=0$ for $j > 0$.

If $p$ does not divide $n$,
then the restriction map from $H^j(A_n,M)$ to $H^j(A_{n-1},M)$ is injective 
by Lemma \ref{restriction}  and the result
follows.  So we may assume that $p|n$.  Since $G$ is perfect,  
Lemma \ref{perfect}
implies that 
$H^1(G,F)=0$,  and since $p$ does not divide the order of the
Schur multiplier of $G$, $H^2(G,F)=0$.
 
As usual, we may assume that $M$ is an irreducible $FG$-module.

If $n=p$, then by Lemma \ref{cyclic}, $\dim H^j(G,M) \le 1$.
 Thus, the result follows by noting that
$\dim M \ge p-2$ for every nontrivial irreducible $G$-module.

Suppose that $n \ne p^a$ for any $a$.  Write $n = p^a + n'$
where $p^a$ is the largest power of $p$ less than $n$.
Set $H=A_{p^a} \times A_{n'} < G$.   Since $H$ contains a
Sylow $p$-subgroup of $G$, it suffices to show by Lemma \ref{restriction} that
$H$ satisfies the conclusion of the theorem.  So let $V=V_1 \otimes V_2$
be an irreducible $FH$-module.  If $V$ is a trivial $H$-module, then
$H^j(H,V)=0$ for $j = 1, 2$ (by Lemma \ref{perfect}).   Otherwise,  the result
follows by Lemma \ref{kunneth} and induction.

Finally suppose that $n=p^{a+1}> p$.  Let $H=A_{p^a} \wr A_p < G$.
Then $H$ contains a Sylow $p$-subgroup of $G$ and again it suffices
to show that the conclusion holds for $H$.   
Let $V$ be an irreducible $FH$-module. Let $N$ be the normal subgroup of $H$
with $H/N \cong A_p$.  So $N=L_1 \times \cdots \times L_p$
where $L_i \cong A_{p^a}$.  The result is straightforward and easier for 
$H^1$, and
we just give the argument for $H^2$.
By Lemma \ref{usual},
$$
\dim H^2(H,V) \le \dim H^2(H/N,V^N) + \dim H^2(N,V)^H + \dim H^1(H/N,H^1(N,V)).
$$
If $N$ is trivial on $V$, then the last two terms are $0$ and the result
holds since we already know the theorem for $n=p$.   So suppose that $V^N=0$.
Let $W$ be an irreducible $FN$-submodule of $V$.  
So $W=W_1 \otimes \cdots \otimes W_p$,
where $W_i$ is an irreducible $FL_i$-module.   By Lemma \ref{kunneth},
$H^1(N,W)=0$ unless (after reordering if necessary) $W_1$ is nontrivial and
$W_j$ is trivial for each $j > 1$.   Suppose that for some $j >1$, 
$W_j,$ is nontrivial. Thus, by Lemma \ref{kunneth},
 $H^1(N,W) \cong H^1(N,gW)= 0$ for every
$g \in G$  Since $W$ is a direct sum of $N$-submodules of
the form $gW, g \in G$, this implies that $H^1(N,V)=0$.   
By Lemma \ref{inflres2}
and induction,  $\dim H^2(H,V) \le \dim H^2(N,V)^H \le (\dim V)/(p-2)$.

Now suppose that $W_1$ is nontrivial and $W_j$ is trivial for all $j > 1$.
Let $W_1 \le M_1$ be the set of fixed points of $L_2 \times \cdots \times L_p$
on $V$.
The stabilizer of $M_1$ is clearly $N_H(L_1)$, 
and so $V$ is induced from $M_1$.
By Shapiro's Lemma (Lemma \ref{shapiro}), $H^2(H,M) \cong H^2(N_H(L_1), M_1)$.
Since $[N_H(L_1):N]$ is prime to $p$, it follows by Lemma \ref{coprime} that 
$\dim H^2(H,M) \le \dim H^2(N,M_1)$.  By Lemma \ref{kunneth} and the fact
that $H^1(L_j,M_1)=0$ for $j > 0$,  $H^2(N,M_1)=H^2(L_1,M_1)$
and the result follows.
\end{proof}

\subsection{ $p=3$.} \hfil\break

\begin{theorem}  Let $G=A_n, n > 2$ and $F$
a field of characteristic $p=3$. Let $M$ be an irreducible $FG$-module.
\begin{enumerate}  
\item If $M$ is trivial, then $H^2(G,M)=0$ unless $n=3,4,6$ or $7$,
in which case $H^2(G,M)$ is $1$-dimensional.
\item If $n = 3^a > 3$, then $\dim H^2(G,M) \le (3/5)\dim M$.
\item If $M$ is nontrivial, then $\dim H^2(G,M) \le  (21/25)\dim  M$.  
\end{enumerate}
\end{theorem}

\begin{proof}  The proof proceeds as in the previous result.
However, note that if $M$ is the trivial module, then
$H^2(G,M)$ is $1$-dimensional for $n=3,4,6$ and $7$ and
$0$ otherwise (for $n \ge 5$, see \cite[p. 314]{gls3} and
for $n=3$ or $4$, a Sylow $3$-subgroup is cyclic). 
We use Lemma \ref{restriction} extensively and sometimes
without comment.

So assume that $M$ is irreducible and nontrivial.
There is no loss in assuming that $F$ is algebraically closed.
  
If $n \le 12$, the result is in  \cite{luxstudent}.
By induction, using Lemma \ref{restriction},
we may assume that $n$ is divisible by $3$
(since $n > 12$ and (1) implies that trivial modules are
not an exception to the bound). \\

Case 1.  $n=3^{a+1} >  9$.\\

Let   $N:=H_1 \times H_2 \times H_3$
where each $H_i = A_{3^a}$.  Let $H$ be a subgroup
of the normalizer of $N$ with $H/N \cong A_3$.   Note that $H$ contains
a Sylow $3$-subgroup of $G$ and so Lemma \ref{restriction} applies.

Let $V$ be an absolutely irreducible $FH$-module.
If $V^N=V$, then by Lemma \ref{usual}, 
$\dim H^2(H,V) \le \dim H^2(H/N,V) + \dim H^2(N,V)^H + \dim H^1(H/N,H^1(N,V))$.
Since $3^a > 3$, the Schur multiplier of $N$ is a $2$-group
and so the middle term above is $0$.  Since $N$ is perfect,
the last term above is $0$, and so $\dim H^2(H,V) \le 1$ by Lemma \ref{cyclic}.

Suppose that $V^N=0$.  Let $W$ be an irreducible $FN$-submodule of $V$.
Write $W=W_1 \otimes W_2 \otimes W_3$ where $W_i$ is an $FH_i$ 
irreducible module.  
We may assume that $W_1$ is nontrivial.
By Lemma \ref{faithful}, either $\dim H^2(H,V) \le \dim H^2(N,V) \le
\dim V/4$ or $W_2$ and $W_3$ are trivial.    Thus, $V=X_N^H$ where 
$X$ is the fixed space of $H_2 \times H_3$.  By Lemma \ref{shapiro},
it follows that $H^2(H,V) \cong H^2(N,X)$.  By Lemma \ref{kunneth}, 
$H^2(N,X) \cong H^2(H_1,X)$.   By induction,   $\dim H^2(H,V) \le (3/5)\dim X
= (1/5) \dim V$.

We claim that the number $d$ of trivial composition factors of $N$ on $M$
is at most $(\dim M)/2$ (in fact, it is usually much less).  
Let $T$ be a Sylow $2$-subgroup of $N$.  It is easy to see that some pair of  
conjugates
of $T$ generate $G$.     So we see by Lemma \ref{trivial cf} that 
$d \le \dim M^T \le (\dim M)/2$.
The previous paragraphs show that
$\dim H^2(H,M) \le d + (\dim M - d)/5$.  Since $d \le (\dim M)/2$, 
this implies
that $\dim H^2(H,M) \le (\dim M)/2 + (\dim M)/10 = (3/5)\dim M$ as required.\\

Case 2.  $n$ is not of the form $3^a + 3$ or $3^a + 6$.  \\

We may assume also that $n \ne 3^a$ (by case 1).  So $n = n_1 + n_2$
where $n_1=3^a > n/3$ and $n_2 \ge 9$.  
Set $H=H_1 \times H_2 < G$ where the $H_i$ are alternating groups
of degree $n_i$.
Note that $[G:H]$ has index prime to $3$.  By induction and
Lemma \ref{kunneth}, the result follows. \\

Case 3. $n=3^a + 3 >  12$.  \\

Let $H=H_1 \times H_2 < G$ with $H_1 = A_3$ and $H_2= A_{3^a}$.
Let $V$ be an irreducible $FH$-module.  Then 
$V$ is trivial for $H_1$ (since it is a normal $3$-subgroup of $H$)
and irreducible for $H_2$. 
If $V$ is nontrivial then, by Lemma \ref{kunneth},
$H^2(H,V) \cong H^1(H_2,V) \oplus H^2(H_2,V)$.  By Theorem \ref{H1alt}
and the fact that $n_1 \ge 27$,
the first term has dimension at most $(1/23) \dim V$ and
the second has dimension at most $(3/5) \dim V$ by induction.  
Thus, $\dim H^2(H,V) < (17/25)\dim M$.

If $V$ is trivial, then $\dim H^2(H,V)=1$.  Arguing as above,
we see that the number of $H$-trivial composition factors in
$M$ is at most $(\dim M)/2$ and so 
$$\frac{\dim H^2(G,M)}{\dim M} \le \frac{1 + 17/25}{2}= \frac{21}{25}.$$

Case 4.  $n=3^a + 6 \ge 15$.   \\

The proof is quite similar to the previous case.

Let  $H=H_1 \times H_2$ where $H_1 \cong A_{3^a}$ and $H_2 \cong A_6$.
Then $H$ contains a Sylow $3$-subgroup of $G$.  So it suffices
to prove the bound for $H$.  Let $V$ be an irreducible
 $FH$-module.  So $V = V_1 \otimes V_2$ where $V_i$ is an irreducible
$FH_i$-module for $i=1,2$.  If both $V_i$ are nontrivial,
then $\dim H^2(H,V) = \dim H^1(H_1,V_1) \cdot \dim H^1(H_2,V_2)$.
The first term is at most $(1/7) \dim V_1$ by Theorem \ref{H1alt}
 and the second is
at most $(1/2) \dim V_2$,  and so $\dim H^2(H,V) \le (1/14)\dim M$.
If $V_1$ nontrivial and $V_2$ is trivial, then by Lemma \ref{kunneth},
$H^2(H,V) \cong H^2(H_1,V_1) \le (3/5) \dim M$.  
If $V_1$ is trivial, then $\dim H^2(A_6,V_2) \le \dim V_2$.
As in the previous case, the number of trivial composition factors for
$H_1$ is at most $(\dim M)/2$,  and the result follows as in the previous
case.   
\end{proof}

\subsection{ $p=2$.} \hfil\break

Let $F$ be an algebraically closed field of characteristic
$2$. In this section, all modules are over $F$.
Let $n \ge 5$ be a positive integer. Write
$n = \sum_{i=1}^r  2^{a_i}$, where the $a_i=a_i(n)$ are decreasing
positive integers.

The next result follows since the $2$-part of the Schur multiplier
for $A_n$ has order $2$ \cite[p. 312]{gls3}.

\begin{lemma} Let $M$ be the trivial module. 
\begin{enumerate}
\item $\dim H^2(A_n,M)=1$; and 
\item $\dim H^2(S_n,M)=2$.
\end{enumerate}
\end{lemma}

\begin{lemma} \label{antosn} Let $M$ be an irreducible nontrivial $FS_n$-module for 
$n \ge 8$.  Then
\begin{enumerate}
\item 
$ \dim H^2(S_n,M) \le \dim H^2(A_n,M) + \dim M/2^{a_1(n)-1}$.
\item $\dim H^2(S_n,M) \le \dim H^2(A_n,M) + \dim M/13$
for $n \ge 12$.   
\end{enumerate}
\end{lemma}

\begin{proof}   
By Lemma \ref{usual}, we have
$$
\dim H^2(S_n,M) \le \dim H^2(A_n,M)^{S_n} + \dim H^1(S_n/A_n, H^1(A_n,M)).
$$
By Theorem \ref{H1alt}, the right hand term is at most $\dim H^1(A_n,M) 
\le \dim M/(f-1)$ where $f$ is the largest prime with $f \le n -2$.  By Bertrand's
postulate, 
there is a prime $f$ with $ a_1(n) \le n/2 \le f -1 \le n-2$, whence (1) holds.
Now (2) holds by the same argument for $n \ge 15$ (since  $13$ is  prime
and $13 \le n-2$),   and by computation
for $n \le 14$ \cite{luxstudent}.  
\end{proof}

\begin{lemma}  Let $n=2^{a+1} \ge 4$.  Let $G=A_n$ or $S_n$.  Let
$F$ be an algebraically closed field of characteristic $2$.
Let $M$ be an irreducible nontrivial $FG$-module.
\begin{enumerate}
\item If $n=4$ and $G=A_n$, then $\dim H^2(G,M)  \le 1$.
\item If $n=4$ and $G=S_n$, then $\dim H^2(G,M)  \le 1$.
\item If $n \ge 8$ and $G=A_n$, then $\dim H^2(G,M) \le (65/24)\dim M$.
\item If $n  \ge 8$ and $G=S_n$, then $\dim H^2(G,M) < (17/6) \dim M$.
\end{enumerate}
\end{lemma}

\begin{proof} If $n \le 8$, this is done by a computer computation in
  \cite{luxstudent}.  So assume that $n \ge 16$. We induct on $n$.
By (2) of the previous lemma, it suffices to prove (3).   
As usual, we will compute cohomology for subgroups which
contain a Sylow $2$-subgroup. So we may use
Lemma \ref{restriction}.

Let $N = A_{2^a} \times A_{2^a}=H_1 \times H_2$ and
$H$ be the normalizer in $G$ of $N$.  
Let $V$ be an irreducible $FH$-module.   
If $V$ is trivial, it follows by Lemma \ref{usual} that
$\dim H^2(H,V) \le \dim H^2(H/N,V) + \dim H^2(N,V)^H + \dim H^1(H/N,H^1(N,V))$.
Since $N$ is perfect, the last term is $0$.  Since $H/N$ has order $4$,
the first term on the right side of the inequality is $3$.   Since the 
Schur multiplier of each factor of $N$ has order $2$,  
$H^2(N,V)$ is $2$-dimensional and $H$ acts nontrivially on this, whence
the middle term has dimension $1$.  Thus,  $\dim H^2(H,V) \le 4$.

Suppose that $V$ is nontrivial.  Then $V^N = 0$. 
 Let $W$ be an irreducible $N$-submodule of $V$.
Write $W=W_1 \otimes W_2$.  Note that $V$ is a direct sum of $N$-modules
of the form $gW$,    $g \in H$.  If both $W_1$ and $W_2$ are nontrivial,
then by Lemma \ref{kunneth}, $H^2(N,W) = H^1(H_1,W_1) \otimes H^1(H_2,W_2)$ and $H^1(N,V)=0$.
 By Theorem \ref{H1alt} and Lemma \ref{usual},  
  $\dim H^2(H,V) \le \dim H^2(N,V)^H \le (\dim V)/36$.

If  $W_2$ is trivial, then $V$ is an induced module -- so
we may write $V=U_J^H$ where $J=N$ or $H/J$ has order $2$.  
Thus, by Lemma \ref{shapiro}, $H^2(H,V) \cong H^2(J,U)$.   If $J=N$, this implies
that $\dim H^2(J,U)=\dim H^2(H_1,U)$.  So by induction,  
$\dim H^2(J,U) \le (65/24) \dim U$, whence $\dim H^2(H,V) \le (65/96) \dim V$.
If $J/N$ has order $2$,  then we apply Lemma \ref{usual}
to conclude that
$$
\dim H^2(J,U) \le \dim H^2(J/N_2,U) + \dim H^2(H_2,U)^J + \dim H^1(J/H_2,H^1(N_2,U)).
$$
Since $H_2$ is perfect and $U$ is trivial for $H_2$, the last term is $0$.
Since $H_2$ is trivial on $U$ and $H^2(N_2,F)$ is $1$-dimensional, we
see that $H^2(N_2,U) \cong U$ as a $J$-module and so $J$ has no fixed points
on the module, whence the middle term is $0$.  Noting that $J/H_2 \cong S_{2^a}$
and using Lemma \ref{antosn} and induction, we conclude that 
$\dim H^2(J,U) \le (17/6) \dim U$. Since $\dim U = (\dim V)/2$, 
we obtain the desired inequality that  $\dim H^2(H,V) \le (17/12) \dim V$.

Note $H$ contains an element $h$ that is the product of two disjoint cycles
of length $2^a-1$.  It is easy to see that $A_n$ can be generated by
two conjugates of $h$.  Setting $J=\langle h \rangle$ and using
Lemma \ref{trivial cf}, we 
see that the number $\alpha$ of trivial $H$-composition factors in $M$
is at most $(\dim M)/2$.  Thus, 
$\dim H^2(H,M) \le 4 \alpha  + (17/12)(\dim M - \alpha)$.  Since $\alpha \le (\dim M)/2$,
this implies that $\dim H^2(H,M) \le (65/24)\dim M$.
 \end{proof}

\begin{theorem}  Let $n > 4$ be a positive integer and 
$F$ an algebraically closed field of characteristic $2$.
Let $M$ be an $FG$-module with $G=A_n$ or $S_n$.
\begin{enumerate}
\item $\dim H^2(A_n,M) \le (35/12) \dim M$; and
\item  $\dim H^2(S_n,M) < 3 \dim M$.
\end{enumerate}
\end{theorem}

\begin{proof}  If $M$ is trivial, then
$\dim H^2(A_n,M)=1$ since a Sylow $2$-subgroup of the
Schur multiplier has order $2$.  Since $S_n/A_n$ has order
$2$, it follows by Lemma \ref{usual} that $\dim H^2(S_n,M) \le 2$, whence
the result holds in this case.  So we may assume that  $M$ is a nontrivial
irreducible $FG$-module.  If $n \le 14$, then (1) and (2)
follow  by  computation \cite{luxstudent}.  So we assume
that $n > 14$.

By Lemma \ref{antosn}, it suffices to prove
the results for $A_n$.  So assume that $G=A_n$ and $M$
is a nontrivial  irreducible $FG$-module.  

The result holds for $n$ a power of $2$ by the previous lemma.
So assume that this is not the case.  As usual, we will obtain
bounds for a subgroup  of odd index and then apply Lemma \ref{restriction}.
 We may also
assume that $n$ is even (since if not, $A_{n-1}$ has odd index).
   Thus, we may write 
$n=n_1 + n_2 $ where $n_1 =2^a$ is the largest power of $2$ less
than $n$ and $n_2 \ge 2$. 

First suppose that $n_2=2$.  Then $S_{2^a}$ contains a Sylow $2$-subgroup
of $G$ and so by the previous lemma, $\dim H^2(G,M) \le (17/6) \dim M$.  
So we assume that $n_2 \ge 4$.  

Let $N=N_1 \times N_2$ where $N_i \cong A_{n_i}$ and let $H$
be the normalizer of $N$.  Then $H/N$ has order $2$ and $H$
contains a Sylow $2$-subgroup of $G$.

Let $V$ be an irreducible $FH$-module.  If $V$ is trivial,
then by Lemma \ref{usual},
$ \dim H^2(H,V) \le \dim H^2(H/N,V) + \dim H^2(N,V)^H + \dim H^1(H/N,H^1(N,V)).$
The first term on the right hand side is $1$.  Since $N=N_1 \times N_2$, 
$H^2(N,V)$ is $2$-dimensional.  If  $n_2 > 4$, then $N$ is perfect,
and $H^1(N,V)=0$.  If $n_2 = 4$, then $H^1(N,V) = \hom(N,V) = 0$.
Thus, $\dim H^2(H,V) \le 3$.

Suppose that $V$ is nontrivial.  Let $W$ be an irreducible
$FN$-submodule.  Write $W=W_1 \otimes W_2$ with the $W_i$ irreducible
$N_i$-modules.   If both of
the $W_i$ are nontrivial, then $H^1(N,W)=0$
by Lemma \ref{kunneth} and so $H^1(N,V)=0$ (since $V$ is a direct
sum of submodules of the form $gW, g \in H$).  
Also by Lemma \ref{kunneth},  $ H^2(N,W) = H^1(N_1,W_1) \otimes H^1(N_2,W_2)$
has dimension at most $(\dim W)/4$ (we leave it to the reader to verify
this  when $n_2=4$).   It follows by
Lemma \ref{usual} that $\dim H^2(H,V) \le (\dim V)/4$.

If $W_1$ is nontrivial, but $W_2$ is trivial, then
 $\dim H^2(N,W) \le \dim H^2(N_1,W) \le (65/24) \dim W$.
By Lemma \ref{usual}, 
$$
\dim H^2(H,V) \le \dim H^2(H/N,V^N) + \dim H^2(N,V)^H + \dim H^1(H/N,H^1(N,V)).
$$
Note that $V^N=0$ and as noted above the middle term is at most $(65/24) (\dim V)$.
Finally, observe that since $H^1(N_2,F)=0$, $\dim H^1(N,V) = \dim H^1(N_1,V)
< \dim V/8$ (this last inequality follows by Theorem \ref{H1alt} for $2^a > 8$).
This gives  $\dim H^2(H,V) \le (17/6) \dim V$.
  
Finally, suppose that $N_1$ is trivial on $V$, but $N_2$ is not.  We
consider the inequality above.  The first  term on the right side of
the inequality is $0$.   If $n_2 > 8$, then the middle term 
is at most $\dim H^2(N_2,V) \le (65/24)\dim V$ and $\dim H^1(N,V)=\dim H^1(N_2,V)
\le (1/13) \dim V$.  Thus $\dim H^2(H,V) \le  ((65/24) + (1/13))(\dim V)< (17/6)\dim V$.
If $n_2 < 8$, it follows by \cite{luxstudent} that
$\dim H^2(N_2,V) + \dim H^1(N_2,V) \le \dim V$ for all nontrivial irreducible
modules.  In particular,  it follows that
$\dim H^2(H,V) \le (17/6)(\dim V)$ for all values of $n_2$.

Arguing as usual, we see that the number of trivial composition factors
of $N_1$ on $M$ is at most $(\dim M)/2$.  Thus, by the previous arguments,
$\dim H^2(H,M) \le (3/2 + 17/12) \dim M = (35/12) \dim M$.
\end{proof}

\section{$\SL$: Low Rank} \label{SLlow}

In this section, we consider the  groups $\SL(d,q), d \le 4$.   We use a gluing
argument and the bounds for $\SL(4,q)$ and $S_d$ to get bounds for
all $\SL(d,q)$ in the next section.

We start with an improvement of the bound given in (1.2) for $H^1$ in the natural 
characteristic.
There are much better bounds for cross characteristic representations
\cite{hoffman}.   We will use Lemma \ref{restriction} without comment below.

\begin{theorem} \label{h1 for sl}  Let $G=\SL(d,q)$ be quasisimple
and $F$ an algebraically closed field of 
characteristic $p$ with $q=p^e$.  Let $M$ be an irreducible $FG$-module.
\begin{enumerate}  
\item If $d=2$, then $\dim H^1(G,M) \le (\dim M)/3$  unless either
$p=2$ and $\dim M = 2$ or $p=3$ and $G=\SL(2,9)$ with $\dim M=4$.
\item $\dim H^1(G,M) \le (\dim M)/d'$,  where $d'$
is the largest prime with $d' \le d$.  
\end{enumerate}
 \end{theorem}

\begin{proof}   If $d=2$, all first cohomology groups have
been computed \cite{hjl}, and the result holds by inspection.

If the center of $G$ acts nontrivially on $M$, then $H^1(G,M)=0$
by Lemma \ref{coprime}.    
So we may view $M$ as an $H$-module with $H=\PSL(d,q)$.
Note that by Lemma \ref{usual},   $H^1(G,M) \cong H^1(H,M)$.

We now prove (2) when $d$ is an odd prime.
Let $T$ be a maximal (irreducible) torus
of size $(q^d-1)/((q-1)\gcd(d,q-1))$.  Let $N=N_H(T)$ and note that
$N/T$ is cyclic of order $d$.  We claim that $d$ does not divide
$|T|$.  If $d$ does not divide $q-1$, this follows
from the fact that $q^d \equiv q \ \mathrm{mod} \ d$.  If $d|(q-1)$,
then write $q = 1 + d^js$ where $d$ does not divide $s$.
Then since $d$ is odd, $q^d = 1 + d^{j+1}s'$ where $d$ does not 
divide $s'$ and this proves the claim.

Let $y \in N$ be of order $d$.   We claim  that $C_T(y)=1$. Since
$\gcd (|T|, d)=1$, we can  lift $T$ to an isomorphic subgroup of $\SL(d,q)$ and
it suffices to prove $C_T(y)=1$  in $\SL(d,q)$.  
Then we can identify $T$ with the norm $1$ elements in $\F_{q^d}^*$
of order prime to $d$
and $y$ induces the $q$-Frobenius automorphism of this field, whence
its fixed point set is $\F_q^*$.   This     has  trivial intersection
with $T$, whence the claim follows.
This also implies that 
$y$ permutes all nontrivial characters of $T$ in orbits of size $d$.
Thus $\dim [T,M]^{\langle y \rangle} = (\dim [T,M])/d$.

Note that up to conjugacy, $y$ is a $d$-cycle in $S_d \le \PSL(d,q)$.
So
$y$ is conjugate to $y^{-1}$ in $S_d$  and so also in $H$.
So choose  $z \in H$ that inverts $y$.  Then $z$ does not normalize
$T$ (since $y$ is not conjugate to $y^{-1}$ in $N$).     
It now follows by the main result of \cite{ber} (based on \cite{gpps}) that 
$H= \langle N, zNz^{-1} \rangle$.  Apply \cite[Lemma 4]{gurkim} to conclude
that $\dim H^1(H,M) \le \dim [T,M]^{\langle y \rangle} \le (\dim M)/d$.

We now complete the proof of  (2) by induction.  We have 
proved the result
for $d$ any odd prime.   It is more convenient work with
$\F_qG$-modules.
Suppose $d$ is not prime -- in particular, $d' \le d-1$ in (2).
Let $P$ be a maximal parabolic stabilizing a $1$-space.  So $P=LQ$
where $Q$ is the unipotent radical of $P$, $L$ is a Levi subgroup with
$L \cong \GL({d}{-1},q)$ and $Q$ is the natural module for 
$J=\SL({d}{-1},q) \le L$.
Since $P$ contains a Sylow $p$-subgroup
of $G$, it suffices to prove the bound for $H^1(P,V)$ with $V$
an irreducible $L$-module.   If $V$ is not isomorphic to $Q$
as $\F_pL$-modules (equivalently, if $V$ is not isomorphic to a Galois
twist of $Q$ as $\F_qL$-modules),  
then $H^1(P,V)=H^1(L,V)$
(by Lemma \ref{usual}) and induction gives the result.   If $V \cong Q$,
then $H^1(L,Q)=0$ unless $q=2$ and $d=4$  (cf. \cite{bell}). This implies
that $\dim H^1(P,V) = 1 \le  (\dim V)/(d-1) \le (\dim V)/d'$. 
If $q=2$ and $d=4$, then $G=A_8$,  and the result is in
\cite{luxstudent}.
\end{proof}

\subsection{$\SL(2)$} \hfil\break

 \begin{theorem} \label{h2forsl2odd}  Let $G=\SL(2,p^e)$ with $p$ odd, 
$p^e \ge 5$
 and $F$ an algebraically
closed field of characteristic $r > 0$. Let $M$ be an irreducible
$FG$-module. Set $q=p^e$.
\begin{enumerate}
\item If $r \ne 2$ and $r \ne p$, then $\dim H^2(G,M) \le 2 (\dim M)/(q-1)$.
\item If $p \ne r=2$, then $\dim H^2(G,M)  \le  \dim M$.
\item If $r=p$, then $\dim H^2(G,M) \le (\dim M)/2$.
\end{enumerate}
\end{theorem}

\begin{proof}  If $q \le 11$, these results all follow
by direct computation \cite{luxstudent}.   So assume that $q > 11$.

In (1), a Sylow $r$-subgroup of $G$ is cyclic,
whence Lemma \ref{cyclic} implies that $\dim H^2(G,M) \le 1$.  
   Since
the smallest nontrivial representation of $\SL(2,q)$ in any
characteristic other than $p$
has dimension $(q-1)/2$, (1) holds.

Now consider (2).  We first bound $\dim H^2(\PSL(2,q), M)$.
Set $H=\PSL(2,q)$.   We refer the reader to \cite{burk} for basic facts about
the $2$-modular representations of $H$.  In particular,  
 every irreducible representation is the reduction of
an irreducible representation in characteristic zero.  The
characters of these representations are described in \cite[Theorem 38.1]{dornhoff}.
 Let $B=TU$ be a Borel subgroup
of $H$ (of order $q(q-1)/2$) with $T$ a torus of order $(q-1)/2$.   Note
that $B$ has a normal subgroup $UT_0$ of index a power of $r$.

It follows from \cite{burk} and \cite[38.1]{dornhoff} that
either $M$ is one of two  Weil modules of dimension $(q-1)/2$  or  has dimension
$q \pm 1$.  Moreover, it follows that the modules of dimension  $q +1$
are all of the form $\lambda_B^H$ with $\lambda$ a nontrivial
$1$-dimensional character of $B$,  and so by Shapiro's Lemma,
$H^2(H,M) \cong H^2(B,\lambda)$.   Since $\lambda$ is nontrivial,
it is nontrivial on $UT_0$ and so by Lemma \ref{coprime},
$H^2(B,\lambda)=0$.     

Suppose that $\dim M=q-1$. If $4|(q-1)$, then $M$ is projective
(cf. \cite[62.3, 62.5]{dornhoff})
and so $H^2(H,M)=0$.   
So suppose that $4|(q+1)$.   By inspection of the character tables
in \cite[38.1]{dornhoff},  $M$ is multiplicity free as a $U$-module and so 
every nontrivial character of $U$ occurs precisely once in $M$.
 Since $T$ acts semiregularly on the nontrivial characters of $U$
(and so all the  $U$-eigenspaces as well),  $M$ is a free rank $2$ module
for the split torus $T$.   Let $x$ be an involution inverting $T$.   Note that
$|T|=(q-1)/2$ is odd.   Thus, the only $T$-eigenspace that
is $x$-invariant is the trivial eigenspace. Write $M=[T,M] \oplus M^T$.  
So all Jordan blocks of $x$ on $[T,M]$ are of size $2$. Since
$\dim M^T = 2$, 
$x$ has either $0$ or $2$ Jordan blocks of size $1$.  Since there
is a unique class of involutions in $H$, this implies that if
$Y$ is any cyclic $2$-subgroup of $H$, then$Y$ has at most
$2$ Jordan blocks of less than maximal size on $M$.  Since Jordan blocks
of maximal size correspond to projective modules, this implies that
$\dim H^k(Y,M) \le 2$ for $k > 0$. 

Let $J$ be a nonsplit torus of order $(q+1)/2$. There is a unique
class of involutions in $H$ and so by conjugating we may
assume that  $x$ is the unique involution in $J$. Write
$J=J_1 \times J_2$ where $J_1$ has odd order and $J_2$
is the Sylow $2$-subgroup of $J$.  Set $s=|J_2|$.
We see from \cite{burk} and \cite{dornhoff} that $\dim M^{J_1} =2s$ or $2s-2$
and so from our observations about the Jordan structure of the involution
in $J$,   $J$ has at most $3$ Jordan blocks with trivial character 
 and at most $2$ of those have size less than $s$.  
By Lemma \ref{coprime}, $H^j(J, [J_1,M])=0$.  So $H^j(J,M)=H^j(J_2,M^{J_1})$.
Thus,  $\dim H^j(J,M) \le 2$ for $j > 0$ and $\dim M^J \le  3$. 
Let $L=N_H(J)$ and note that
$[H:L]$ is odd.   By Lemma \ref{usual}, 
$$
\dim H^2(L,M) \le \dim H^2(L/J,M^J) + \dim H^2(J,M) + \dim H^1(L/J,H^1(J,M)).
$$
The first term on the right is at most $3$, the middle term at most $2$
and the last term at most $2$.  Thus,  $\dim H^2(G,M) \le 7$.  Since $q > 11$, 
this implies that $\dim H^2(G,M) < (7/18) \dim M < (\dim M)/2$.   

Finally, consider the case that $M$ is a Weil module.  In this case
(by \cite[Theorem 38.1]{dornhoff}), $Q$ has $(q-1)/2$ distinct characters
on $M$ that are freely permuted by $T$, whence 
$M \cong FT$ as an $FT$-module.
   If $4|(q-1)$, then the normalizer of $T$
has odd index, and arguing as above,  we see that $\dim H^2(H,M) \le 1$.
If $4$ does not divide $q-1$, then as above, we see that an involution has
precisely one trivial Jordan block, and so if $i > 0$, $\dim H^i(J,M) \le 1$.   In all cases
 $\dim H^2(H,M) \le 3$.   Thus, $\dim H^2(H,M) < (\dim M)/2$ (since $q \ge 11$).

Let $Z=Z(G)$ and note that $|Z|=2$. 
In all cases,
$$
\dim H^2(G,M) \le \dim H^2(H,M) + \dim H^2(Z,M)^G + \dim H^1(H,H^1(Z,M)).
$$
If $M$ is trivial, then $\dim H^2(G,M)=1$.   Otherwise, the middle term on the right is
$0$.   The first term on the right is at most $\dim M / 2$.  The 
last term is at most $(\dim M)/2$.   Thus, $\dim H^2(G,M) \le \dim M$.

Finally, consider (3).  It is more convenient to work over
$F=\F_p$ in this case. Let $B=TU$ be a Borel subgroup
with $|U|=q$.  

Let $W$ be an irreducible $FB$-module.   Then by Lemma \ref{usual} and
the fact that $T$ has order coprime to $p$,  
$$
\dim H^2(B,W) \le \dim H^2(U,W)^T.
$$

Note that $U$ is a $T$-module by conjugation.  
As we have seen (Lemma \ref{H2 for abelian}), 
$H^2(U,W)= \wedge^2(U^*) \otimes W \oplus U^* \otimes W$.
So $T$ has fixed points if and only if either $U \cong W$  or $W$ is a homomorphic
image of  $\wedge^2(U)$.  Note we are taking exterior powers over $F$ and so
$\wedge^2(U)$ has dimension $e(e-1)/2$.  Suppose that $\alpha$ of order
$(q-1)/2$ is an eigenvalue on $U$ for a generator $t$ of $T$.
Then the eigenvalues of $t$ on $U$ (over the algebraic closure) are just the 
$e$ Galois conjugates of $\alpha$. 
So the eigenvalues of $t$ on $\wedge^2(U)$ are all Galois conjugates of
$\alpha^{1 + p^j}$ for some $0< j < e$.  Note that this is never a Galois conjugate
of $\alpha$ and also $T$ has no multiple eigenvalues on $\wedge^2(U)$
(by Lemma \ref{multiplicity free}).

So we have seen that $H^2(B,W)=0$ unless 
$W \cong U$ 
or $W \cong U \otimes_{\F_q} U^{p^j}$ for some $j$ with $1 < j < e$
as $\F_pT$-modules. 
Moreover, as noted above,  $\wedge^2(U)$ is multiplicity free, and so,
using Lemma \ref{H2 for abelian}, $\dim H^2(B,W)^T \le \dim W$.

This already gives the inequality  $\dim H^2(G,M) \le \dim H^2(B,M) \le \dim M$.
If $M$ is trivial, then $H^2(G,M)=0$.  If $M$ is irreducible and nontrivial, then
if $W$ occurs as a $T$-submodule, so does 
$W^*$. Thus, in fact, $\dim H^2(B,M) \le (\dim M)/2$.~\end{proof}

\begin{theorem} \label{h2forsl2eve}  Let $G=\SL(2,q)$ with $q=2^e \ge 4$, 
 and let $F$ be an algebraically
closed field of characteristic $r > 0$. Let $M$ be an irreducible
$FG$-module.
\begin{enumerate}
\item If $r \ne 2$, then $\dim H^2(G,M) \le ( \dim M)/(q-1)$.
\item If $r=2$, then $\dim H^2(G,M) \le (\dim M)/2$ unless $2^e=4$
and $M$ is the trivial module.  
\end{enumerate}
\end{theorem}

\begin{proof}  If $r$ is odd, then a Sylow $r$-subgroup is cyclic, whence
$\dim H^2(G,M) \le 1$ by Lemma \ref{cyclic}.  If $M$ is trivial, $H^2(G,M)=0$. 
  It is obvious that the smallest faithful representation
for a Borel subgroup has dimension $q-1$, whence also for $G$. 

Let $r=2$.  If $q > 4$, the proof is identical to the proof in
the previous lemma  when $p=r$.  
 If $q=4$, then $\wedge^2(U)$
is the trivial module, which explains why $H^2(\SL(2,4),F) \ne 0$.
For $q=4$, the result follows by \cite{luxstudent}.   
\end{proof}

\subsection{ $\SL(3)$} \hfil\break

\begin{theorem} \label{sl3 lemma}  Let $p$ and $r$ be primes.  Let $G=\SL(3,q),q=p^e$, $F=\F_r$
and $M$ an irreducible $FG$-module.  Then either $\dim H^2(G,M) \le  \dim M$
or $3=r \ne p$ and   $ \dim H^2(G,M) \le (3/2)\dim M$.
\end{theorem}

\begin{proof}  If $q \le 4$, the result follows by a direct computation
\cite{luxstudent}. So assume that $q > 4$.
  
First consider the case that $r \ne p$.

If $r$ does not divide $q-1$, then 
a Sylow $r$-subgroup of $G$ is cyclic, whence the result holds
by Lemma \ref{cyclic}.

Next suppose that $3 \ne r|(q-1)$.
Then  a Sylow $r$-subgroup fixes a $1$-space and a complementary
$2$-space in the natural representation.
   Let $P$ be the full
stabilizer of one of these subspaces with unipotent radical $Q$.  
Then for one of the choices for $P$, 
$\dim M^Q \le (\dim M)/2$.   Note that $P=LQ$ with $L=\GL(2,q)$.
Let $J=\SL(2,q) < L$. 
So $\dim H^2(G,M) \le \dim H^2(P,M) \le \dim H^2(L,M^Q)$
(here we are using Lemma \ref{usual} and the fact that $r$ 
does not divide $|Q|$).
Let $V$ be an irreducible $L$-module. 
If $V$ is the trivial module, then $\dim H^2(L,V)=1$
(since $J$ has trivial Schur multiplier and $L/J$ is cyclic of
order a multiple of $r$).
Otherwise, by Lemma \ref{usual}, 
$$
\dim H^2(L,V) \le \dim H^2(J,V) + 
\dim H^1(L/J, H^1(J,V)).
$$
By the results of the previous subsection,
$\dim H^2(J,V) \le \dim V$, and by (1.2)
$\dim H^1(J,V) \le (\dim V)/2$. 
   Thus,
$\dim H^2(G,M) \le  (1.5)\dim M^Q \le (3/4) \dim M$.

Suppose that $r=3$ does divide
$q-1$.   
Then a Sylow $3$-subgroup is contained in the normalizer $H$ of a
split torus $T$ and $H/T \cong S_3$.  

Let $V$ be an irreducible $FH$-module.  If $T$ is nontrivial on $V$,
then by Corollary \ref{coprime},  $H^2(H,V)=0$.  So we may assume
that $\dim V=1$ and $H$ either acts trivially on $V$
 or via the sign representation
for $S_3$.  
Now we use Lemma \ref{usual} to
see that 
$$
\dim H^2(H,V) \le \dim H^2(S_3,V) + \dim H^2(T,V)^H + \dim H^1(S_3,H^1(T,V)).
$$
Note that $H^1(T,V) \cong \hom(T,V)$ is an indecomposable $2$-dimensional $S_3$ module.
Thus, by Lemma \ref{cyclic},  the left and right hand terms of the right side
of the above inequality are each at most $1$.      
Finally, by Lemma \ref{H2 for abelian}, 
there is an exact sequence
$$
0 \rightarrow \Hom_H(T,V) \rightarrow H^2(T,V)^H \rightarrow \Hom_H(\wedge^2(T),V).
$$
Note that the only $H$-simple homomorphic image of $T$ is the sign representation
for $S_3$ while $\wedge^2(T)$ only surjects onto the trivial module.
Thus, $\dim H^2(T,V)^H \le 1$ and so $\dim H^2(H,V) \le 3 \dim V$. 
  Let $T_0$ denote the Hall $3'$-subgroup
of $T$.   Then $M=[T_0,M] \oplus M^{T_0}$.   Since $q > 4$,  $T_0$
is nontrivial (indeed it is either a Klein group of order $4$ or contains
a regular semisimple element).  Considering the maximal subgroups of 
$\SL(3,q)$ \cite{mitchell} and \cite{hartley},  
there are two conjugates of $T_0$ which generate $G$, whence
$\dim M^{T_0} \le (\dim M)/2$ by Lemma \ref{trivial cf}.  Since $H^2(H,[T_0,M])=0$ by Lemma 
\ref{coprime}, the computation above shows that
$\dim H^2(H,M) \le 3 \dim M^{T_0} \le (3/2) (\dim M)$.
 
Now consider the case $p=r$.   
Let $P=LQ$ be the stabilizer
of a $1$-space or a hyperplane where $Q$ is the unipotent radical of
$P$ and $L \cong \GL(2,q)$ is a Levi complement.
Let $Z=Z(L)$ and note that $Z$ is cyclic of order $q-1$. 
Let $T$ be a split torus containing $Z$ (of order $(q-1)^2$).
 
Let $V$ be an irreducible $FP$-module.
It suffices to prove that $\dim H^2(P,V) \le \dim V$.  
By Lemma \ref{usual},  
$$
\dim H^2(P,V) \le \dim H^2(L,V) + \dim H^2(Q,V)^L + \dim H^1(L,H^1(Q,V)).
$$
Consider the middle term on the right.  Using Lemma \ref{H2 for abelian}
and arguing as in the proof of 
Lemma \ref{multiplicity free},
$\wedge^2 Q$ has distinct composition factors
as an $L$-module (and no composition factor is  isomorphic to $Q$ as an $L$-module), 
whence the second term has dimension at most
$(\dim V)/2$ (since the dimension of $V$ is at least $2$ over $\mathrm{End}_G(V)$
unless $[L,L]$ acts trivially on $V$ but then $V$ is not a homomorphic image
of $Q$ or $\wedge^2(Q)$).

Suppose that $Z$ is trivial on $V$.  Since  $q > 4$, $Z$ is nontrivial
on $Q$ and every composition factor of $\wedge^2 Q$, whence the middle
term on the right is $0$.  Similarly, $Z$ acts without fixed points on $H^1(Q,V)$
and so by Corollary \ref{coprime},
$H^1(L,H^1(Q,V))=0$ as well.   Thus,  $\dim H^2(P,V) \le \dim H^2(L,V)$
and this is at most $(\dim V)/2$ by the result for $\SL(2)$.

Now suppose that $Z$ is nontrivial on $V$.  By Lemma \ref{coprime}, $H^2(L,V)=0$.  
Note that $W:=H^1(Q,V) \cong Q^* \otimes V$ as an $L$-module (however, the tensor
product is over $\F_p$).    Since $V$ is $L$-irreducible, it must
be $Z$-homogeneous.  
By Lemma \ref{coprime}, $H^1(L,W)=H^1(L,W^Z)$ and this is either $0$
unless $V$ and $Q$ involve the same $\F_pZ$-irreducible module.
If that is the case,  then $\dim W^Z = 2 \dim V$.   Thus
by (1.2),  $\dim H^1(L,W) \le \dim V$.   In this case,  $\mathrm{Hom}_Z(\wedge^2 Q,V)=0$
and so $\dim H^2(P,V) \le \dim V$ unless perhaps $Q \cong V$.  
We still obtain this inequality since we compute in this case
that $\dim H^1(L,W) < \dim V$.

If $Z$ is nontrivial on $V$, but $V$ and $Q$ do not involve the
same irreducible $\F_pZ$-module, then $H^1(L,W)=0$ and so in
this case $\dim H^2(P,V) \le (\dim V)/2$.    
 \end{proof}

\subsection{ $\SL(4)$} \hfil\break

\begin{theorem} \label{sl4 natural}
Let $G=\SL(4,q)$, $q = p^e$.  Let $F=\F_p$.  If $M$ is an irreducible
$\F_pG$-module, then $\dim H^2(G,M) < 2 \dim M$.
\end{theorem}

\begin{proof}  If $q \le 3$,
see  \cite{luxstudent}.  So assume that $q > 3$.
 Let $P$ be the stabilizer of a $1$-space.  
Since $P$ contains a Sylow $p$-subgroup, it suffices to 
bound $\dim H^2(P,M)$.   Write $P=LQ$ where $Q$ is the unipotent
radical of $P$ and $L \cong \GL(3,q)$ is the Levi complement.

  By Lemma \ref{usual}
$$
\dim H^2(P,M) \le \dim H^2(L,M^Q)+ \dim H^2(Q,M)^L + \dim H^1(L,H^1(Q,M)).
$$

Since $G$ is generated by $Q$ and the radical of the opposite parabolic,
by replacing $P$ by its opposite, we may assume that $\dim M^Q \le (1/2) (\dim M)$,
and so by the result for $\SL(3)$, $ \dim H^2(L,M^Q) \le (1/2) \dim M$.  

Consider $H^2(Q,M)$ as an $L$-module.   By taking a $P$-composition series
for $M$, it suffices to bound $\dim H^2(Q,V)^L$ where $V$ is an irreducible
$FP$-module.   By Lemma \ref{H2 for abelian}
we have the  exact  sequence, 
$$
0 \rightarrow \hom(Q,V)^L \rightarrow H^2(Q,V)^L \rightarrow \hom(\wedge^2(Q),V)^L
$$

Since $Q$ is an irreducible $FL$-module,  the second term either is zero
or is isomorphic to $\mathrm{End}_{L}(Q) \cong \F_q$ and so has dimension  $(\dim V)/3$.   
Next consider $\wedge^2(Q)$ over $\F_p$.  
Note that $Q \otimes_{\F_p} \F_q$ is 
 is the sum of Galois twists $Q_i, 1 \le i \le e$ as an $\F_qL$-module.
The exterior
square will be the sum  of all $Q_i \otimes Q_j, i < j$ plus the sum of all
$\wedge^2 Q_i$.  We note that these are all irreducible and nonisomorphic
as $\F_qL$-modules.  Since none of them is isomorphic to $Q$, it follows
that if $V = Q$, then $\dim H^2(Q,V)^L \le (\dim V)/3$. 
So assume this is not the case.
   Thus, $\hom(\wedge^2(Q),V)^L=0$
unless $V$ is isomorphic to one of
 $Q_i \otimes Q_j$ or $\wedge^2(Q_i)$   If $V$ is one of these modules,
then $\hom(\wedge^2(Q),V)^L \cong \mathrm{End}_L(V)$ and so has dimension
at most $(\dim V)/3$.      Thus, $H^2(Q,V)^L$ has dimension at most $(\dim V)/3$.

Finally, consider the far right term of the sequence
above.   Again, we can take a $P$-composition
series for $M$ and consider $H^1(L,H^1(Q,V))$ for $V$ an irreducible $FP$-module.
Now $H^1(Q,V)=\Hom(Q,V)$ (over $\F_p$).   Let $T=Z(L)$.  By Corollary \ref{coprime},
$H^1(L,H^1(Q,V)) = H^1(L,H^1(Q,V)^T)$.   Then $\dim H^1(Q,V)^T \le 3 \dim V$.
So applying Theorem \ref{h1 for sl}, we  see that 
$\dim H^1(L,H^1(Q,V)) \le \dim V$.    Thus,
$\dim H^2(G,M)/\dim M < 1/2   + 1/3 + 1 < 2$
as required.
\end{proof}

\begin{theorem} \label{sl4 cross}  Let $G=\SL(4,q),
q = p^e$.  Let $F=\F_r$ for $r$ a prime.   If $M$ is an irreducible
$FG$-module, then $\dim H^2(G,M)  < 2 \dim M$.
\end{theorem}

\begin{proof}  By the previous result we may assume that $r \ne p$.

 Let $R$ be a Sylow $r$-subgroup.  We consider
various cases.

First suppose that $r > 3$, whence $R$ is abelian.
   If $r|(q-1)$, then $R \le J$, the 
monomial group $J:=T.S_4$ where $T$ is a split torus.   
Since $r$ does not divide $|S_4|$, Lemma \ref{usual} and Corollary \ref{coprime}
imply that 
$\dim H^2(J,M) \le \dim H^2(T,M)^{S_4}$.  It suffices to prove
the inequality for $W$ irreducible for $J$.  If $T$ is not trivial
on $W$, then $H^2(J,W)=0$.  If $T$ is trivial on $W$, then
by Lemma \ref{H2 for abelian}, 
$\dim H^2(T,W) \le \dim \hom_G (T,W) + \dim \hom_G(\wedge^2(T),W)$.
Note that  the only irreducible quotient of $T$ is the $3$-dimensional
summand of the permutation module.  Similarly, the only irreducible quotient
of $\wedge^2(T)$ is the same module.   So if $W$ is not that module, 
$H^2(J,W)=0$.  If $W$ is that module, then each term on the right
in the above inequality is $1$ and so $\dim H^2(J,W) \le 2 < 3 = \dim W$.

If $r >3$ and does not divide $q-1$,   then $R$ has rank at most $2$.
If $R$ is cyclic, the result holds by Lemma \ref{cyclic}.   If not, then
$R$ is contained in the stabilizer of a $2$-space.   
The radical $Q$ of this parabolic has fixed space of dimension at most $(\dim M)/2$
(since $Q$ and its opposite generate $G$).
Lemma \ref{usual}  together with the fact that $\dim H^2(R,F) =3$ gives
  $\dim H^2(R,M) \le (3/2) (\dim M)$,  and the result holds.

 Suppose $r=3$.  If $3$ does not divide $q-1$, then $R$ is abelian
 and stabilizes a $2$-space and the argument above applies. 
   So suppose that $3|(q-1)$.  Then $R$ fixes
a $1$-space and $3$-space,  and so is contained in
the corresponding parabolic $P=QL$.   We may assume
that $\dim M^Q \le  (\dim M)/2$.   Thus,
$\dim H^2(P,M) \le \dim H^2(L,M^Q)$.  

By the result for $\SL(3,q)$ (Theorem \ref{sl3 lemma}),  we see that 
$H^2(J,M^Q)$ has dimension at most $ (3/2) (\dim M^Q)$ where
$J$ is the derived subgroup of $L$.   So
$$
\dim H^2(L,M^Q) \le \dim H^2(L/J,M^{QL}) + \dim H^2(J,M^Q) + \dim H^1(L/J,H^1(J,M^Q)).
$$
The terms on the right are  bounded by $(\dim M)/2$, $(3/4) \dim M$ 
and $(\dim M)/4$, whence $\dim H^2(G,M) < 2 \dim M$.

Finally, consider the case $r=2$.  If $q=3$, see  \cite{luxstudent}.  So assume that $q \ge 5$. 
We work over an algebraically closed field.  If $M$ is the trivial
module, then $\dim H^2(G,F)= 0$ (since the Schur multiplier of $\SL(4,q)$ is
trivial).  So assume that $M$ is not trivial.
By computing orders, we see that 
$R$ is contained in $H$, the stabilizer of a pair of complementary $2$-spaces.
Let   $L = L_1 \times L_2 = \SL(2,q) \times \SL(2,q)$.  Note that
$L$ is normal in $H$,  and $H/L$ is a dihedral group of order
$2(q-1)$.  

Let $V$ be an irreducible $H$-module.  If $V^L = 0$, let $W$ be an
irreducible $L$-submodule of $V$.  So $W=W_1 \otimes W_2$ with
$W_i$ an irreducible $L_i$-module.    If
each $W_i$ is nontrivial, then by Lemma \ref{kunneth}, 
$H^1(L,V)=0$, and so  by Lemma \ref{inflres2},  it follows that
$\dim H^2(H,V) \le \dim H^2(L,V) \le (\dim V)/4$.  
If $W_1$ is nontrivial and $W_2$ is trivial, then $V$ is an induced
module, and so $H^2(H,V) \cong H^2(X,D)$ where $X/L$ is cyclic and
$L_2$ is trivial on $D$.   Then by Lemma \ref{usual}, 
$$
\dim H^2(X,D) \le \dim H^2(L,D) + \dim H^1(X/L,H^1(L,D)).
$$
Then $\dim H^2(L,D) = \dim H^2(L_1,D) \le \dim D \le (\dim V)/2$.
Furthermore, we know that 
 $\dim H^1(X/L,H^1(L,D)) \le \dim H^1(L,D) = \dim H^1(L_1,D)$.  
 By (1.2),  $\dim H^1(L_1,D)  \le (\dim D)/2 
 \le  (\dim V)/4$.  Thus, $\dim H^2(H,V) \le (3/4) (\dim V)$.  

If $V$ is trivial for $L$, then by Lemma \ref{usual},
$\dim H^2(H,V) \le \dim H^2(H/L,V) + \dim H^2(L,V) + \dim H^1(H/L,H^1(L,V))$.
Since $L$ is perfect and since the Schur multiplier of $L$ is trivial, it follows
by Lemma \ref{perfect} that $\dim H^2(H,V) \le \dim H^2(H/L,V)$.
Since $H/L$ is dihedral,   by Corollary \ref{coprime} it follows
that either $H^2(H/L,V)=0$ or
$V$ is trivial.   It is easy to see that $H^2(H/L, \F_2)$ is $3$-dimensional.
It is straightforward to see that $G$ can be generated by two conjugates
of an odd order subgroup of $L$, whence $H$ can have at most $(\dim M)/2$
trivial composition factors by Lemma \ref{trivial cf}. Thus,  
$\dim H^2(H,M)/\dim M \le (3/2) + (3/8) < 2$. 
\end{proof}

\section{$\SL$:  The General Case}  \label{SLgen}

We handle $\SL(n)$ by means of a gluing argument.    
This is a variation of the presentations given in \cite{pap1} and \cite{pap3}.
Note also that the proposition below applies in either 
the profinite or discrete
categories.  The key idea is that it suffices to check relations on
subgroups generated by pairs of simple root subgroups -- this is a
consequence of the Curtis-Steinberg-Tits presentation (see \cite{curtis}).
We will also use this method to deduce the result for groups of rank at
least $3$ from the results on alternating groups and  on rank $2$ groups.

We state the Curtis-Steinberg-Tits result in the following form:

\begin{lemma} \label{curtis}  Let $G$
be the universal Chevalley group of a given type
of rank at least $2$ over a given field.  Let
$\Pi$ be the set of simple positive roots of the
corresponding Dynkin diagram and let $L_{\delta}$
be the rank one subgroup of $G$ generated by the root subgroups
$U_{\delta}$ and $U_{-\delta}$ for $\delta \in \Pi$.
    Let $X$ be a group generated
by subgroups $X_{\delta}, \delta \in \Pi$.    
Suppose that $\pi: X \rightarrow G$
is a homomorphism such that $\pi(X_{\delta})=L_{\delta}$ and
$\pi$ is injective on $\langle X_{\alpha}, X_{\beta} \rangle $ for each
$\alpha, \beta \in \Pi$.   Then
$\pi$ is an isomorphism.  
\end{lemma}

Let $X$ and $Y$ be two disjoint sets of size $2$.
Set $G=\SL(n,q)=\SL(V)$ for $n > 4$.  Let $e_1, \ldots, e_n$
be a basis for $V$.  Let 
$\langle X | R \rangle$ be a presentation for $A_n$ 
(acting on the set $\{1, \ldots, n\}$) and 
$\langle Y | S \rangle$  a presentation for $\SL(4,q)$
acting on a space $W$ that is the span of $e_1, \ldots, e_4$
(viewing these either as profinite presentations or discrete presentations).

  Let $G_1$ be the subgroup of
$G$ consisting of the elements which permute the elements of the basis as
even permutations.      Let $G_2$ be the subgroup of $G$ 
that acts trivially on
$e_j, j > 4$ and preserves the subspace generated by $e_1, \ldots, e_4$.
Let $L$ be the subgroup of $\SL(4,q)$ leaving the span of
$\{e_1, e_2\}$ invariant and acting trivially on $e_3$ and $e_4$.
So $L \cong \SL(2,q)$.  
Let $S \cong A_4$ be the subgroup of $\SL(4,q)$ consisting of the even
permutations of $e_1, \ldots, e_4$.   Pick generators 
$u,v$ of $S$ where $u = (e_1 \ e_2)(e_3 \ e_4)$ and $v = (e_1 \ e_2 \ e_3)$.   
Note that $u$ normalizes $L$.   Choose $a \in L$
such that $L = \langle a, a^u \rangle$ 
(e.g., we can take $a$ to be almost any element of order $q+1$).

Let $T \cong A_4$ be the subgroup of $A_n$ fixing all $j > 4$.
In $T$, let $u'=(12)(34)$ and $v'=(123)$.   Let $K \cong A_{n-2}$
be the subgroup of $A_n$ fixing the first two basis vectors.
Let $b$ and $c$ be any generators for $K$.

Let $J$ be the group generated by $X \cup Y$ with relations 
$R,S$,  $u=u'$,  $v=v'$, $[a,b]=[a,c]=1$.   Let $J_1 \le J$
be the subgroup generated by $X$,  and $J_2$ the subgroup
generated by $Y$.

There clearly is a homomorphism $\gamma:J \rightarrow G$ determined
by sending $J_i$ to $G_i$ for $i=1,2$ (where we send $X$ to 
the corresponding permutation matrices in $G$ and $Y$ to
the corresponding elements in $G_2$ -- all relations in $J$
are satisfied and so this gives the desired homomorphism).
In particular, this shows that $J_i \cong G_i$ for $i=1$ and $2$
and so we may identify $G_i$ and $J_i$.   In particular, $u$ and
$v$ are words in $Y$ and $u',v'$ are words in $X$.

\begin{proposition} $J \cong G$.  \end{proposition}

\begin{proof} As we noted above, there is a surjection $\gamma:J \rightarrow G$ that
sends $J_1$ to $G_1$ and $J_2$ to $G_2$.
It suffices to show that $\gamma$ is an isomorphism.
We also view $a,b$ and $c$ as elements of $J$, and $L$ as a subgroup of $J$.

We first show that $[K,L]=1$ in $J$.   By the relations, we have
that $[a,K]=1$.   Since $u'$ normalizes $K$ and  
  $u=u'$, we see that $1=[a^u, K^{u'}]=[a^u,K]$.  Since $L= \langle a, a^u \rangle$,
$[K,L]=1$.   Set
$E:=\langle K, u' \rangle \cong S_{n-2} \le A_n$.
  Since $u$ normalizes $L$, we see that
$E$ does as well.
Note that $E$ is precisely the stabilizer in $A_n$ of the subset
$\{1,2\}$.
This is a maximal subgroup of $A_n$, and since $A_n$ does not normalize
$L$ (since $\gamma(A_n)$ does not normalize $\gamma(L)$ in
$G$), it follows that $E=N_{A_n}(L)$ (in $J$).

Let $\Omega$ be the set of conjugates of $L$ under $A_n$ in $J$.
By the previous remarks, $|\Omega|=[A_n:S_{n-2}]=n(n-1)/2$ and moreover,
there is an identification between $\Omega$ and the subsets of size
$2$ of $\{1, \ldots, n\}$.  Let $L_{i,j}$ denote the conjugate
of $L$ corresponding to the subset $\{i,j\}$.   Note that
$\gamma(L_{i,j})$ is the subgroup of $G$ that preserves the
$2$-space $\{e_i, e_j\}$ and acts trivially on the other basis vectors
of $V$.   

Let $\Delta$ be the
orbit of $L$ under $A_4$.  Note that $|\Delta|=6$ and $\Delta$ corresponds to
the two element subsets of $\{1,2,3,4\}$.  Since $A_n$
is a rank $3$ permutation group on $\Omega$, any pair of distinct
conjugates of $L$ in $\Omega$  is conjugate to either the pair $\{L,L_{2,3}\}$ or
$\{L,L_{3,4}\}$.   

Suppose that $L_1$ and $L_2$ are two
of these conjugates.   By the above remarks, they are conjugate by
some element in the group to 
$L$ and $M=L^x$ for some $x \in A_4$.  In particular, we see that
$M$ is the subgroup of $\SL(4,q)$ fixing the $2$-space generated 
by $e_{x(1)}$ and $e_{x(2)}$ and and fixing the vectors
$e_{x(3)}$ and $e_{x(4)}$.   Since we are now inside $\SL(4,q)$, we see that
 either $[L,M]=1$ or $L$ and $M$ generate an $\SL(3,q) \le \SL(4,q)$.
Since $\gamma$ is injective on $\SL(4,q)$, $\gamma$ is injective
on $\langle L, M \rangle$ and so is injective on
the subgroup generated by $\langle L^{h_1}, L^{h_2} \rangle $ for 
any elements $h_1, h_2 \in A_n$.   

Thus, by the Curtis-Steinberg-Tits relations 
(Lemma \ref{curtis}),  $N= \langle \{L^g | g \in A_n \} \rangle \cong
G$,  and indeed $\gamma:N \rightarrow G$ is an isomorphism.   

It suffices to show that $J=N$.  Since $A_n$ normalizes $N$ and
since $\SL(4,q) \le N$ ($\SL(4,q)$ contains the $A_4$ conjugates of $L$
and these generate $\SL(4,q)$),
it follows that $N$ is normal in $J$.    Clearly, $\SL(4,q)$ is trivial
in $J/N$ and since $A_n \cap N \ge A_4$, it follows that
$A_n \le N$ as well.  Thus, $J=N$ and the proof is complete.
\end{proof}
 
Since $L$ and $K$ are $2$-generated, $J$ is presented
by $4$ generators and $|R|+|S| + 4$ relations.

By Theorem \ref{H2alt}, we 
have profinite presentations for $A_n$ with $4$ relations.
By  Theorems \ref{sl4 natural} and \ref{sl4 cross},  
$\SL(4,q)$ has a profinite presentation with $3$ relations. 
Thus we have a profinite presentation
for $\SL(n,q)$ with $4$ generators and $4 + 3 + 4=11$ relations.
Using Lemma \ref{saving}, we obtain:

\begin{corollary} \label{sln corollary} Let $G=\SL(n,q)$ with $n \ge 5$.
Then $G$ has a profinite presentation on $2$ generators and $9$ relations.
In particular, $\hat{r}(G) \le \ 9$.
\end{corollary}

\begin{theorem} \label{sln theorem}  Let $G$ be a quasisimple
group that surjects on $\PSL(n,q)$.  Let $F$ be a field. Then 
\begin{enumerate} 
\item $\hat{r}(G) \le 9$; and
\item $\dim H^2(G,M) \le 8.5 \dim M$ for any  $FG$-module $M$.
\end{enumerate}
\end{theorem}

\begin{proof}   If $\SL(n,q)$ has trivial Schur multiplier, then (1)
follows by Corollary \ref{covering relations} and the previous result.
 This is the case
unless $(n,q)=(2,4),(2,9), (3,2),(3,4)$ or $(4,2)$ \cite[p. 313]{gls3}.
In those cases, we have a smaller value for $\hat{r}(\SL(n,q))$
and Corollary \ref{covering relations} gives (1).  Now (2) follows
from (1) by (1.4). 
\end{proof}

\section{Low Rank Groups} \label{low rank}

In this section, we consider the rank one and rank two finite groups
of Lie type.  We also consider some of the rank three groups which
are used for our gluing method.

The method for the low rank groups is fairly straightforward.  With more
work, one can obtain better bounds.    As usual, we will use
Lemma \ref{restriction} without comment.   We first consider the rank one
groups.

\begin{lemma}  \label{rank one}
Let $G$ be the universal cover of a rank one simple finite group of Lie type.
\begin{enumerate}
\item If $G=\SL(2,q)$, then $h(G) \le 1$.
\item  If $G=\Sz(q), q = 2^{2k+1} > 2$, then $h(G) \le 1$.
\item  If $G=\SU(3,q),  q > 2$, then $h(G) \le 2$.
\item  If $G=\Re(q), q=3^{2k+1} > 3$, then $h(G) \le 3$.
\end{enumerate}
\end{lemma}

\begin{proof}  Let $R$ be 
a Sylow $r$-subgroup of $G$ for some prime $r$. 
Let $F=\F_r$. 
 
(1) is proved in the previous section and
(2) is proved in \cite{wilson}.

 Consider $G=\SU(3,q)$ with $q=p^e$.  First suppose
that $p \ne r$.  If $r  \ne  3$, then $R$
is either cyclic or stabilizes a nondegenerate
subspace and so embeds in $\GU(2,q)$.  
We use the result for $\SL(2,q)$  and Lemma \ref{usual} 
to deduce the result. 

If $3$ does not divide $q+1$, the above argument
applies to $r=3$.  Suppose that $r=3|(q+1)$.   
Then $R$ is contained in the stabilizer
of an orthonormal basis and we argue precisely
as  we did for $\SL(3,q)$ in Theorem \ref{sl3 lemma}.

So assume that $p=r$ and $R \le B$, a Borel subgroup.
Write $B=TR$ with $T$ cyclic of order $q^2-1$.
Let $Z=Z(R)$ of order $q$.  If $q=4$, one computes
directly that the bound holds.  So assume that $q > 4$.  Then   
$T$ acts irreducibly on $Z$ and on $R/Z$.   
By Lemma \ref{usual}, for $V$ an irreducible $\F_pB$-module (i.e. a $T$-module),
$$
\dim H^2(B,V) \le \dim H^2(B/Z,V) + \dim H^2(Z,V)^B + \dim H^1(B/Z,H^1(Z,V)).
$$
Similarly,  
$\dim H^2(B/Z,V) \le \dim H^2(R/Z,V)^T$.   Since $V$ is a trivial $R$-module,
$H^2(R/Z,V)^T=0$ unless $V \cong R/Z$ or $V$ is a constituent of 
$\wedge^2(R/Z)$.  We argue as usual to show that
$\wedge^2(R/Z)$ is multiplicity free (and does not surject onto $R/Z$).
It follows that   $\dim H^2(R/Z,V)^T \le \dim \mathrm{End}_T(V)
\cong V$ (as vector spaces).  The same argument shows that either $ H^2(Z,V)^B=0$
or $V \cong Z$ or $V$ is a constituent of $\wedge^2(Z)$, and in those cases
$ H^2(Z,V)^B \cong V$ (as vector spaces).   Finally, note that 
$H^1(B/Z,H^1(Z,V)) \cong H^1(B/Z, V^*)$, and so is either $0$ or has dimension
equal to $\dim V$ if $V^* \cong R/Z$.   So we see that each term is at most $\dim V$,
and at most two of them can be nonzero.  Thus, $\dim H^2(B,V) \le 2 \dim V$.

Finally, consider $G=\Re(3^{2k+1}), k > 1$.   
See \cite{suzuki1, suzuki2} for properties of $G$.

   If $r=2$, then $R$ is contained in
$H:=C_2 \times \PSL(2,q)$.  Let $V$ be a nontrivial irreducible $FH$-module.
By Lemma \ref{kunneth},  $H^2(H,V) \cong H^2(\PSL(2,q))$.
Similarly,  if $V$ is trivial, Lemma \ref{kunneth} implies that 
$H^2(H,V) \cong H^2(C_2,F) \oplus H^2(\PSL(2,q),F)$ and so is $2$-dimensional.
By Lemma \ref{h2forsl2eve}, it follows that $\dim H^2(H,V) \le 2 \dim V$.  
 If $r > 3$, then $R$ is cyclic
and the result holds by Lemma \ref{cyclic}.
 If $r=3$, then a Borel subgroup is $TR$ where
$T$ is cyclic of order $q-1$.  Moreover, there are normal $T$-invariant subgroups
$1 =R_0 < R_1 < R_2 < R_3=R$ such that $T$ acts irreducibly on each
successive quotient (acting faithfully on the first and last quotients and acting
via a group of order $(q-1)/2$ on the middle quotient).  Furthermore, $R_2$
is elementary abelian.    Let $V$ be an irreducible $B$-module.
Then,  by Lemma \ref{usual}, 
$$
\dim H^2(B,V) \le \dim H^2(B/R_2,V) + \dim H^2(R_2,V)^B + \dim H^1(B/R_2,H^1(R_2,V)),
$$
and 
$$
\dim H^2(B/R_2,V) \le \dim H^2(R_3/R_2,V)^B.
$$

By Lemma \ref{multiplicity free} and Lemma \ref{H2 for abelian}, 
it follows that $\dim H^2(R_3/R_2,V)^B \le \dim V$
and $\dim H^2(R_2,V)^B \le \dim V$.  Finally, consider the final term
on the right.   We can write $R_2=W(\alpha) \oplus W(\beta)$ as a direct
sum of the $T$-eigenspaces with characters  $\alpha$
and $\beta$ (of orders
$q-1$ and $(q-1)/2$).  Write $V=W(\gamma)$ as a $T$-module.
Then $H^1(R_2,V)$ is a direct sum of modules $W(\alpha^{-1}\gamma')$ and 
$W(\beta^{-1}\gamma')$
where $\gamma'$ is a Frobenius twist of $\gamma$.  Since $T$ has order
coprime to the characteristic, we see that
$\dim H^2(B/R_2,V)$ will be the multiplicity of $R_3/R_2 = W(\alpha)$
in $H^1(R_2,V)$.   The comments above show that this multiplicity is $0$
 unless  $\gamma$ is a 
product of two twists of $\alpha$ or a twist of $\alpha$ times a twist of $\beta$.
It follows that the multiplicity in these cases is $1$ and $\dim H^2(B,V) \le \dim V$.
Thus, 
$\dim H^2(B,V) \le 3 \dim V$ as required.
\end{proof}

We now consider the groups of rank $2$, subdividing them into
two classes.  The classical groups of rank $2$ will arise in the consideration
of higher rank groups and so we need better bounds.  The remaining cases
do not occur as Levi subgroups in higher rank groups and so do not
impact any of our gluing arguments.

\begin{lemma} \label{rank 2 classical}
\begin{enumerate}
\item If $G=\SL(3,q)$, then $h(G) \le  3/2$.
\item If $G=\SU(4,q)$, then $h(G) \le 9/4$.
\item If $G=\SU(5,q)$, then $h(G) \le 4  $.
\item If $G=\Sp_4(q)$, then $h(G) \le 3 $.
\end{enumerate}
\end{lemma}
 
\begin{proof}

We handle the various groups separately proving somewhat
better results.   The result for $\SL(3,q)$ is a special case
of Theorem \ref{sl3 lemma}.   
  Let $\F_q$ be the field
of definition of the group with $q=p^e$.   Let $r$ be a prime, $R$
be a Sylow $r$-subgroup of $G$ and $F=\F_r$.  If $M$ is a trivial $FG$-module, the
result is clear (because we know the Schur multiplier \cite[pp. 312--313]{gls3}).
So  it suffices to consider nontrivial irreducible $\F G$-modules. \\

\noindent Case 1.  $G= \SU(4,q)$.  \\

If $r \ne p$ does not divide $q$,  then the argument is identical
to that given for $G=\SL(4,q)$.   Suppose that $r=p$.
Let $P$ be the stabilizer of a totally singular $2$-space.  So $P=LQ$
where $L=\GL(2, q^2)$ and $Q$ is an irreducible $\F_qL$-module of order $q^4$.
By Lemma \ref{usual}, 
$$
\dim H^2(G,M) \le \dim H^2(L,M^Q) + \dim H^2(Q,M)^L + \dim H^1(L,H^1(Q,M)).
$$
If $q=2, 3$, we apply \cite{luxstudent}.  So assume that $q > 3$.  
Let $Z=Z(L)$.   By Corollary \ref{coprime}, 
$H^1(L, H^1(Q,M))= H^1(L, H^1(Q,M)^Z) = H^1(L, \mathrm{Hom}_Z(Q,M))$.
Since $\dim H^1(\SL(2,q),W) \le (\dim W)/2$ by (1.2), one can see
that $\dim H^1(L,H^1(Q,M)) \le \dim M$.  By Theorems \ref{h2forsl2odd} and \ref{h2forsl2eve}, 
 $\dim H^2(SL(2,q),W) \le  (\dim W)/2 $ and so $\dim H^2(L, M^Q) \le (\dim M^Q)/2 \le (\dim M)/4$.      
Since $\wedge^2(Q)$ is multiplicity free
(arguing exactly as in Lemma \ref{multiplicity free}), the
middle term is certainly at most $\dim M$ and so
$\dim H^2(G,M) \le (9/4) (\dim M) $.  \\

\noindent Case 2.  $G= \SU(5,q)$.\\

First consider the case $r \ne p$ and $r > 2$.
If $r$ does not divide $q+1$, then either $R$ is cyclic
or $R$ embeds in $\SU(3,q)$ or $\SU(4,q)$ and the result
follows.

Suppose that $r|(q+1)$.   
Then $R$ is contained in $H$, the stabilizer of an orthonormal basis.
In particular, $H$ has a normal abelian subgroup $N$ that is homogeneous
of rank $4$ with $H/N = S_5$.  Let $V$ be an irreducible
$FH$-module.  By Lemma \ref{usual}, 
$$
\dim H^2(H,V) \le \dim H^2(S_5, V^N) + \dim H^2(N,V)^{S_5} + \dim H^1(S_5,H^1(N,V)).
$$
If $N$ acts nontrivially on $V$, then Lemma \ref{coprime} implies that
$H^2(H,V)=0$.  So assume that this is not the case.  

Since $S_5$ has a cyclic Sylow $r$-subgroup, the first term on the right is
at most $1$ by Lemma \ref{cyclic}.   Since $N$ does not have a $1$-dimensional
quotient (as an $S_5$-module), 
it follows that $\dim H^2(H,\F_r) \le 1$.  So we may
assume that $\dim V > 1$, and so $\dim V \ge 3$.

Recall that $\dim H^2(N,V)^{S_5}  \le \dim  \hom_{S_5}(N,V) + \dim \hom_{S_5}(\wedge^2 N,V)$.  
So if $V$ is a not a quotient of either $N$ or $\wedge^2(N)$, then
$\dim H^2(H,V) \le 1 \le (1/3)\dim V$.   So assume that $V$
is a quotient of either $N$ or $\wedge^2(N)$.

If $r=5$, the only quotients of $N$ and $\wedge^2(N)$ are $3$-dimensional.
Since $\dim N=4$ and $\dim \wedge^2(N)=6$, it follows that
$\dim H^2(H,V) \le 4 = (4/3)\dim V$.  

So assume that $r \ne 5$.   If $V$ is a quotient of $N$, then
$V$ is the irreducible summand of the permutation module for $S_5$.
Thus, $H^2(S_5,V)=0$ by Lemma \ref{shapiro}.  
By dimension, it is clear that $\dim  \hom_{S_5}(N,V) + \dim \hom_{S_5}(\wedge^2 N,V) \le 2$,
whence the result.  

If $V$ is a nontrivial  quotient of $\wedge^2(N)$ and is not a quotient of $N$,
then the same argument shows that $\dim H^2(H,V) \le 3$.  

$H^2(S_5,V)=0$ if $V$ is $1$-dimensional and $\dim V \ge 3$ otherwise, 
this implies that $\dim H^2(S_5,V^N) \le (\dim V)/3$.
This same argument shows that $\dim H^1(S_5,W) \le (\dim W)/3$
for any $FS_5$-module and so $ H^1(S_5,H^1(N,V))
\le (\dim N)(\dim V)/3 \le (4/3) (\dim V)$.   

Consider the case that $r=2 \ne p$.   Then $R \le H:=\GU(4,q)$.  We use
the result for $N:=\SU(4,q)$ and Lemma \ref{usual}.  So
$\dim H^2(G,M) \le \dim H^2(H/N, M^N) + \dim H^2(N,M)^H + \dim H^1(H/N,H^1(N,M))$.
This gives $\dim H^2(G,M) \le 4 \dim M$ as above.

Finally, consider the case that $r=p$. Let $P$ be the stabilizer of a totally singular
$2$-space.   Then $P=LU$ where $L$ is the Levi subgroup of $P$ and 
$U$ is the unipotent radical.  Let $J=\SL(2,q^2) \le L \cong \GL(2,q^2)$
and $Z=Z(L)$ cyclic of order $q^2-1$.
Also,  note  that  $W=[U,U]$ is irreducible of order
$q^4$ and $X:=U/W$ is an irreducible $2$-dimensional module
(over $F_{q^2})$ and that $W$ is an irreducible $4$-dimensional module
over $\F_q$ -- it is isomorphic to $X \otimes X^{(q)}$ (which is defined over
$\F_q$).  

Let $V$ be an irreducible $FP$-module.  
By Lemma \ref{usual},  
$$\dim H^2(P,V) \le \dim H^2(P/W ,V) + \dim H^2(W,V)^P + \dim H^1(P/W,H^1(W,V)).
$$

Consider the first term  on the right hand side of the inequality.  
By Lemma \ref{usual}, 
$\dim H^2(P/W,V) \le \dim H^2(L,V) + \dim H^2(U/W,V)^L + \dim H^1(L,H^1(U/W,V))$.
Note that $P/W$ is very similar to a maximal parabolic subgroup of $\SL(3,q^2)$.
Arguing precisely as in that case, we see that $\dim H^2(P/W,V) \le \dim V$.

Now consider the middle term.  It is straightforward to see
(arguing as in the proof of Lemma \ref{multiplicity free})  that $\wedge^2(W)$
is multiplicity free and has no composition factors isomorphic to $W$, whence
the middle term has dimension at most $\dim V$.   

Finally, consider the last term on the right. 
Set $Y=H^1(W,V) \cong W^* \otimes V$.  By Lemma \ref{h1 nonfaithful}, 
$\dim H^1(P/W,Y) \le \dim H^1(L,Y) + \dim \hom_L(U/W,Y)$.
By Corollary \ref{coprime},  $H^1(L,Y) \cong H^1(L,Y^Z)$.
Note that $\dim Y^Z \le 2 \dim V$ and so by (1.2), it follows
that $\dim H^1(L,Y) \le \dim H^1(J,Y^Z) \dim V$. 
Thus,  $\hom_L(U/W,Y) \cong \hom_L (U/W \otimes W, V)$.   

Let $\lambda$ be the fundamental dominant weight for $J$.   
So $U/W = X=X(\lambda)$ (the natural module over $\F_{q^2}$).
Note that $U$ is an $\F_qJ$-module satisfying 
$U \otimes_{\F_q} \otimes F_{q^2} \cong  X \otimes_{\F_{q^2}} X^{(q)}$.   
   It is straightforward to see that $U/W \otimes W$ 
modulo its radical is multiplicity free.    
Thus 
 $\hom_L(U/W,Y)$ is either $0$ or is isomorphic to $\mathrm{End}_L(V)$,
 and so has dimension at most $\dim V$.   It follows that $\dim H^2(P,V) \le 4 \dim V$.
 \\

\noindent Case 3.  $G=\Sp(4,q)$.  \\

If $r > 3$ and $r \ne p$, then $R$ is abelian of rank $2$, whence
$\dim H^2(G,M) \le 3 \dim M$ by Lemma \ref{H2 for abelian}. 

If $ 3 \ge r\ne p$, then $R$ is contained in $J$, the stabilizer of
a pair of orthogonal nondegenerate $2$-spaces.
If $r=3$, this implies that 
$$
\dim H^2(G,M) \le \dim H^2(\SL(2,q) \times \SL(2,q), M) \le  \dim M$$
 by \S \ref{SLlow} and Lemma \ref{kunneth}.
If $r=2$,  then this shows that $\dim H^2(J',M) \le \dim M$.
By Lemma \ref{usual}, 
$$
\dim H^2(J,M) \le \dim H^2(J/J', M^{J'}) + \dim H^2(J',M)^J
+ \dim H^1(J/J',H^1(J',M)),
$$
and so $\dim H^2(J,M) < 3 \dim M$.

If $p=r$, then $R \le P$, the stabilizer of a totally singular $2$-space.
Write $P=LQ$ where $L$ is a Levi subgroup and $Q$ the unipotent radical.
Note $Q$ is elementary abelian of order $q^3$. 
By Lemma \ref{usual}, 
$$
\dim H^2(P,M) \le \dim H^2(L,M^Q) + \dim H^2(Q,M)^L + \dim H^1(L,H^1(Q,M)).
$$
The first term on the right is at most $\dim M^Q$ (by the result for $\SL(2,q)$)
and is at most $(\dim M)/2$.   Arguing as for $\SL(3,q)$, 
$\dim H^1(L,H^1(Q,M)) \le (3/2)\dim M$.  
By Lemma \ref{H2 for abelian},  the middle term is at most $\dim M$,
 whence $\dim H^2(G,M) < 3 \dim M$.
\end{proof}

We now consider the remaining rank $2$ groups.  

\begin{lemma} \label{rank 2 exceptional}
Let $G$ be a quasisimple finite group of Lie type and rank $2$.  
Then $\hat{r}(G) \le 6$.
\end{lemma}

\begin{proof}  By the preceding lemma, we may assume that
$G$ is one of $\G_2(q), {^3}D_4(q)$ or ${^2}\FFF(q)'$.   
 Let $p$ be the prime dividing $q$.  Note that $p=2$ in
the last case.   

Since $\G_2(2) \cong \PSU(3,3)$, we assume that $q > 2$ if $G=\G_2(q)$.
We also note that a presentation is known for ${^2}\FFF(2)'$ which gives
the result (cf. \cite{rwilson}), so we also assume that $q > 2$ in that case.

Let $r$ be a prime, $F$ a field of characteristic $r$ and $M$ an irreducible
$FG$-module.  Let $R$ be a Sylow $r$-subgroup of $G$.

\noindent Case 1.  $G=\G_2(q), q > 2$.\\

If $r \ne p$ and $r > 3$, then $R$ is contained in a maximal
torus (since the order of $R$ is prime to the order of the
Weyl group) and so $R$ is abelian 
of rank at most $2$, whence $\dim H^2(G,M) \le 3 \dim M$
by Lemma \ref{H2 for abelian}.    If $p \ne r \le 3$,
then $R$ is contained in $L$ with 
$L \cong \SL(3,q).2$ or $\SU(3,q).2$ (for example, noting that the only
prime dividing the indices of both of these  subgroups is $p$).  If $r \ne 2$, the result
follows from the corresponding result for $L$.  If $r=2$,  by Lemma \ref{usual}
$\dim H^2(L,V) \le \dim H^2(L/J,V^J) + \dim H^2(J,V)^L + \dim H^1(L/J,H^1(J,V))$,
where $J$ is the derived subgroup of $L$ and $V$ is an $FL$-module.
If $V$ is trivial, then this gives  $\dim H^2(L,V) \le 1$.  Otherwise,  $V^J=0$, and
$\dim H^2(L,V) \le \dim H^2(J,V) + \dim H^1(J,V) < 4 \dim V$. 
So $\dim H^2(G,M) \le  4 \dim M$.

Now assume that $r=p$.   Let $R \le P$ be a maximal parabolic subgroup.
Write $P=LQ$ where $L$ is a Levi subgroup and $Q$ is the unipotent radical.
We may choose $P$ so that $Q$ has a normal subgroup $Q_1$ with
$Q/Q_1$ and $Q_1$ each elementary abelian (of dimension $2$ or $3$
over $\F_q$).  Let $V$
be an irreducible $FP$-module.  It suffices by Lemma \ref{restriction}
to prove the bound in the lemma for $P$. 

 By Lemma \ref{usual}, 
$$
\dim H^2(P,V) \le \dim H^2(L,V) + \dim H^2(Q,V)^L + \dim H^1(L,H^1(Q,V)).
$$
Note that $X:=H^1(Q,V) = \hom (Q,V)=\hom(W,V)$, where $W=Q/Q_1$
has order $q^2$.   If $q$ is prime, then $\dim X \le 2 \dim V$,
and so by (1.2), $\dim H^1(L,X) \le \dim V$.  If $q$ is not prime,
then $Z=Z(L)$ acts nontrivially on $W$ and so 
$H^1(L,X)=H^1(L,X^Z)$ by Corollary \ref{coprime}.   
Since $\dim X^Z \le 2 \dim V$, the same bound
holds in this case.   By the result  for $\SL(2,q)$ (Theorems \ref{h2forsl2odd}
and \ref{h2forsl2eve}), $\dim H^2(L,V) \le (1/2)(\dim V)$.
So to finish this case, it suffices to show that $\dim H^2(Q,V)^L \le  (5/2)(\dim V)$.

By Lemma \ref{usual}, 
$$
\dim H^2(Q,V) \le \dim H^2(Q/Q_1,V) + \dim H^2(Q_1,V) + \dim H^1(Q/Q_1,H^1(Q_1,V)).
$$
The proof of this inequality (either using a spectral sequence or more directly
in \cite{holt1}) shows that we have the same inequality after taking $L$-fixed points.
Using Lemma \ref{H2 for abelian} and arguing as usual, we see that 
the sum of the first two terms on the right is at most $\dim V$.
Similarly, the right-most term is $\hom(Q/Q_1 \otimes Q_1^*,V)$
and the dimension of the $L$-fixed points is at most $\dim V$.   The result follows.  \\

\noindent Case 2.  $G={^3}D_4(q)$.\\

First suppose  $r \ne p$.   If $p \ne r > 3$, then $R$ is  contained in
a maximal torus and is abelian.  By inspection, $R$ has
rank at most $2$ and so $\dim H^2(G,M) \le 3 \dim M$.
If $r=3$,  then $R \le H$, the central product of  $\SL(2,q) \circ SL(2,q^3)$, whence
we can use the bounds in \S \ref{SLlow} (obtaining a bound of $4 \dim M$).
If $r=2$, then $R  \le N_G(H)$ and $H$ has index $2$ in $N_G(H)$.
The bound for $H$ shows that $\dim H^2(N_G(H), M) \le 5 \dim M$.

If $r=p$, then $R \le P=LQ$ with $P$ a maximal parabolic, $Q$ its
unipotent radical and 
$L$ a Levi subgroup with simple composition factor $\SL(2,q^3)$.
Then $|Z(Q)|=q$ and $Q/Z(Q)$ is the tensor
product of the three twists of the natural module for $\SL(2,q^3)$
(over $\F_q)$).
We argue as in the previous case to see that $\dim H^2(G,M) \le 5 \dim M$.  \\

\noindent Case 3.  $G={^2}\FFF(q), q > 2'$. \\

First suppose that  $r> 3$. Then $R$ is abelian
of rank at most $2$ (by inspection of the maximal
tori -- see \cite{malle}),  and so  by Lemma \ref{H2 for abelian}
$\dim H^2(G,M) \le 3 \dim M$.

Note that $G$ contains a subgroup $H \cong SU(3,q)$
(see \cite{malle}). 
If $r=3$, then $R \le H$ and so
$\dim H^2(G,M) \le \dim H^2(\SU(3,q), M) \le 3 \dim M$
by Lemma \ref{rank one}.

If $r=p=2$, then $R \le P=LQ$ with $P$ a maximal parabolic, 
$L = \Sz(q) \times C_{q-1}$ its Levi subgroup  
and unipotent radical $Q$.   There is a sequence of normal
subgroups $Q_1 < Q_2 < Q$ with elementary abelian quotients of order 
$q, q^4$ and $q^5$ respectively. We argue as above and conclude
that $\dim H^2(G,M) \le 5 \dim M$.
\end{proof}

We consider two families of rank three groups that are used
in the bounds for $\FFF(q)$ and ${^2}\EE_6(q)$.

\begin{lemma}  \label{some rank 3}
If $G=\Sp(6,q)$ or $\SU(6,q)$, then $h(G) \le 6$.
\end{lemma}

\begin{proof}  The proofs are similar to the rank $2$ cases and
since the bounds are quite weak, we only sketch the proof.

Let $F$ be a field
of characteristic $r$.  Let $R$ be a Sylow $r$-subgroup of $G$.

First consider $G=\Sp(6,q)$ with $q=p^a$.   

If $p \ne r \ge 5$, then $R$ is abelian of
rank at most $3$, whence $\dim H^2(R,M) \le 6 \dim M$
by Lemma \ref{H2 for abelian}.

If $r =3 \ne p$, then $R$ is a contained in the stabilizer of
a totally singular $3$-space and so $R \le \GL(3,q)$ and the
result follows by the result for $\SL(3,q)$ and the standard
argument.   If $r=2 \ne p$,
then $R \le \Sp(4,q) \times \Sp(2,q)$ and we argue as usual.

If $r=p$, then $R \le P$, the stabilizer of a totally isotropic
$3$-space.  Then $P=LU$ where $L=\GL(3,q)$ is the Levi subgroup
and $U$ is elementary abelian of order $q^6$ 
(and irreducible for $L$ when
$q$ is odd).   We argue as usual.  

Now suppose that  $G=\SU(6,q)$ with $q=p^a$.  
First consider the case that $p \ne r$. 
If $r >3$ does not divide $q+1$, then  $R$ is abelian of rank at most
$3$, whence the result holds.  If $r=3$ does not divide $q+1$,
then $R \le \GL(3,q^2)$ and the result
follows.

 If $3 \le r$ does divide $q+1$,
then  $R \le S:=A.S_6$ where $A$ is isomorphic
to $C_{q+1}^5$.  The result now follows by using the bounds
for $S_6$ and Lemma \ref{usual}.

If $r=2 \ne p$, then $R$ stabilizes a nondegenerate $4$-space.
So we use the results for $\SU(2,q)$ and $\SU(4,q)$
and argue as usual.

If $r=p$, then $R \le P$, the stabilizer of a totally singular $3$-space.
Note $P=LQ$ where $Q$ is the unipotent radical and $L$ the Levi
subgroup.  Note $Q$ is an irreducible $L$-module of order $q^9$
and $L \cong \GL(3,q^2)$.   We argue as usual. 
\end{proof}

\section{Groups of Lie Type  -- The  General Case}   \label{classical}

Here we essentially follow the argument in \cite{pap1}
but use profinite presentations.

\begin{theorem} \label{classical theorem}
Let $G$ be a quasisimple  finite group with $G/Z(G)$ a group of
Lie type. 
Then $\hat{r}(G) \le 18$.
\end{theorem}

\begin{proof}
 Let $G$ be the simply connected group of the given type
of rank $n$. By the results of the previous section,
we may assume that $n \ge 3$.  Consider the Dynkin diagram for $G$.
Let $\Pi$ be the set of simple roots and write 
$\Pi=\{ \alpha_1, \ldots, \alpha_n\}$.

First suppose that $G$ is a classical group.  We assume the
numbering of roots is such that the subsystem $\{ \alpha_1, \ldots, \alpha_{n-1}\}$
is of type $A_{n-1}$, $\alpha_n$ is an end node root and is connected
to only one simple root $\alpha_j$ in the Dynkin diagram
(in the typical numbering for a Dynkin diagram, $j=n-1$ except for
type $D$ when $j=n-2$).

Let $G_1$ be the subgroup generated by
the root subgroups $U_{\pm \alpha_i}, 1 \le i < n$.   
Let $G_2$ be the rank $2$
subgroup generated by the root subgroups  
$U_{\pm \alpha_n}, U_{\pm \alpha_j}$.  Let $L_2$ be the rank
$1$ subgroup corresponding to the simple root $\alpha_n$. 
Let $L_1$ be the subgroup of $G_1$ generated by the root subgroups
that commute with $L_2$.  Note that $L_1$ is an $\SL$  unless
$G$ has type $D_n$ in which case $L_2$ is of type $\SL(2) \times \SL(n-2)$.
Let $L$ be the
rank one subgroup generated by  $U_{ \pm \alpha_j}$.   
Let $\langle X | R \rangle$ be a presentation for $G_1$
and $\langle Y |  S  \rangle$ be a presentation for $G_2$
with $X$ and $Y$ disjoint.

We give a presentation for a group $J$ with generators
$X \cup Y$ and relations $R,S$, $[L_1,L_2]=1$ and we identify
the copies of $L$ in $G_1$ and $G_2$.  More precisely, 
take two generators for each $L_i$, express them as words
in $X$ and $Y$ and impose the four commutation relations. 
Similarly, take our two generators for $L$ and take two words each
in $X$ and $Y$ which map onto those generators of $L$ in $G$
and equate the corresponding words.

We claim $J \cong G$.  Clearly, $J$ surjects onto $G$.  Thus,
the subgroup generated by $X$ in this presentation can
be identified with $G_1$ and the subgroup generated by $Y$
can be identified with $G_2$.  Now $J$ is generated by the
simple root subgroups contained in $G_1$ or $G_2$.  Any two
of the these root subgroups (and their negatives) satisfy
the Curtis-Steinberg-Tits relations (for either they are both
in $G_1$ or $G_2$ or they commute by our relations since  $[L_1, L_2]=1$).  
By Lemma \ref{curtis}  $J$ is a homomorphic
image of the universal finite group of Lie type  of the given type,
and the claim follows.  

Note that the number of relations is $|R|+|S| + 6$ (since $4$ relations
are required to ensure that $[L_1,L_2]=1$ and $2$ relations to identify
the copies of $L$) and the number of generators
is $|X| +|Y|$.     Using
Lemma \ref{saving} and the fact that 
$G, G_1$ and $G_2$ are all $2$-generated,  
we see that 
$$
\hat{r}(G) \le \hat{r}(G_1) + \hat{r}(G_2) + 6 - 2.
$$

Now $G_1 \cong \SL$  and so satisfies $\hat{r}(G_1) \le 9$ by 
Corollary \ref{sln corollary}, 
and $G_2$ is either of type $B_2$ or $\SU(d,q)$ with $d=4$ or $5$.   In particular,
$\hat{r}(G_2) \le 5$ by Lemma \ref{rank 2 classical} and (1.4).
This gives $\hat{r}(G) \le 18$ as required, 
and also $\hat{r}(G/Z) \le 18$ for any central subgroup $Z$ of $G$ by
Corollary \ref{covering relations}.  

We now consider the exceptional groups.
The idea is essentially the same, but
we  have to modify the construction slightly.
If $G=\EE_n(q)$ with $6 \le n \le 8$,  $G_1$ still has type $A_{n-1}$,
$G_2$ has type $A_2$, but $L_1=A_2 \times A_{n-4}$, and $L_1$
is generated by $2$ elements, there is no difference in the 
analysis of the presentation.  Thus,  $\hat{r}(G) \le 18$.

If $G=\FFF$, we  take $G_1=C_3$ and $G_2=A_2$.
Then $L_1=A_1$.
Similarly, if $G={^2}\EE_6(q)$, then $G_1=\SU(6,q)$ and $G_2$ is of type $A_2$.
By Lemma \ref{some rank 3}, $\hat{r}(G_1) \le 7$.  Since $\hat{r}(G_2) \le 3$,
we see that $\hat{r}(G) \le 3 + 7 + 2 + 4 - 2 = 14$.  

Now let $G$ be a quasisimple group with $G/Z(G)$ a simple finite group
of Lie type.  If $G$ is a homomorphic image the universal Chevalley
group, then we have shown that in all cases $\hat{r}(G) \le 18$.
We need to consider the possibility that $G$ has a Schur multiplier whose
order divides the characteristic of $G$.  If $G/Z(G)$ is isomorphic
to  an  alternating group, we have already proved the result.
By  \cite[p. 313]{gls3} the only groups $G/Z(G)$ that remain to be considered
are $\PSL(3,2),  \PSL(3,4), \PSU(4,2), \PSU(6,2),$ \
 $ \Sp(6,2),
 \Sz(8), \POmega^+(8,2),\G_2(4), \FFF(2)$ and ${^2}\EE_6(2)$.
In all these cases, we have shown that $\hat{r}(G/Z) \le 14$ 
Thus, $\hat{r}(G) \le 15$ by Corollary \ref{covering relations}.
\end{proof}

\section{Sporadic Groups} \label{sporadic}

Now let $G$ be a quasisimple sporadic group
and $M$ an irreducible $\F_pG$-module.   In this section, we prove:

\begin{theorem} \label{sporadic theorem}  
Let $G$ be a finite quasisimple group
with $G/Z(G)$ a sporadic simple group.  Then
  $G$ has a profinite presentation with $2$ generators and $18$ relations, and
 $\dim H^2(G,M) \le (17.5)\dim M$ for any $FG$-module $M$.
\end{theorem}

  One can certainly
prove better bounds.   We use the main result of Holt \cite{holt} to 
see that $\dim H^2(G/Z,M) \le 2e_p(G/Z)\dim M$,  where $p^{e_p(G)}$
is the order of a Sylow $p$-subgroup of $G$.  Also, for many
of the groups, there is a presentation with less than $18$
relations (see \cite{rwilson}), whence the results follow
(note that in all cases the Schur multiplier is cyclic \cite[p. 313]{gls3}).

So we only need to deal with those sporadic groups (and their
covering groups) where neither of these arguments suffices.
The only cases to consider are $p=2$ and a  few cases
for $p=3$.  

In these cases, it is more convenient to work with the simple group
rather than the covering group.

Table 1 lists the cases that are not covered by Holt's
result or by the presentations given in \cite{rwilson}.
We give the structure of a subgroup $H$ of the simple
group $S:=G/Z$ that contains a Sylow $p$-subgroup of $S$
in order to apply Lemma \ref{restriction}.

Let $G=Co_1$ and let $N=O_p(H)$.   
 
Note that   a Sylow $2$-subgroup of $M_{24}$
is contained in a subgroup isomorphic to $2^4A_8$.   Using the results
for $A_8$ and the computations in \cite{luxstudent}, we see that
$\dim H^2(M_{24}, M) \le \dim M$.
The standard arguments now yield
$\dim H^2(H,M) \le 3 \dim M$ for $M$ an $\F_2H$-module where
$H = 2^{11}M_{24}$, and therefore we obtain the same bound for $G$.
Similar computations using the subgroups
in the table show that the results hold  in all the remaining cases.\\

 \begin{table}
 \centerline{TABLE 1}
$$
\begin{array}{lccc}
G/Z     & |Z|_p  &  p & H       \\
\hline &&& \\
Co_3    &  1 & 2 &      2^4\cdot A_8 \\
Co_2    &  1 & 2 &     2^{1+8}\cdot Sp(6,2)\\
Co_1    &  2 & 2 &     2^{11} \cdot M_{24} \\
He      &  1 & 2 &     2^6 \cdot 3.S_6    \\
Fi_{22} &  2 & 2 &     2^{10}\cdot M_{22} \\
Fi_{23} &  1 & 2 &     2\cdot Fi_{22} \\
Fi_{24}'&  1 & 2 &     2^{11} \cdot M_{24} \\
Suz     &  2 & 2 &      2^{1+6}\cdot U_4(2)\\
J_4     &  1 & 2 &     2^{11} \cdot M_{24} \\
HN      &  1 & 2 &     2^{1+8}\cdot (A_5 \times A_5).2       \\    
Th      &  1 & 2 &     2^5\cdot L_5(2)                \\
B       &  2 & 2 &     2^{1 + 22} \cdot Co_2 \\
M       &  1 & 2 &     2^{1+24} \cdot Co_1  \\
\hline &&& \\
Fi_{23}  & 1  & 3  &   O^+(8,3)\cdot 3\\
Fi_{24}' & 3  & 3  &   3^{1+10}\cdot U(5,2) \\
B        & 1  & 3  &   3^{1+8}\cdot 2^{1+6}\cdot U(4,2) \\
M        & 1  & 3  &   3^8 \cdot O^{-}(8,3)\\
\hline &&&
\end{array}
$$
\end{table}

By (1.1), this completes the proof of Theorem \ref{sporadic theorem}.

\section{Higher Cohomology}  \label{higher cohomology}

We have seen that  $\dim H^k(G,M) \le C \dim M$ for $M$ a faithful
irreducible $FG$-module and $k \le 2$.  
In fact, it is unknown whether there is an absolute
bound $C_k$  for $\dim H^k(G,M)$ for $M$ an absolutely irreducible 
$FG$-module with
$G$ simple.  It was conjectured by the first author over 
twenty years ago that this
was the case for $k=1$.  Indeed, there are no examples known
with  $\dim H^1(G,M) > 3$ for $M$
an absolutely irreducible $FG$-module and $G$ a 
finite simple group.   So we ask again:

\begin{question}  \label{absolute bound}
 For which $k$ is it true that there is an absolute
constant $C_k$ such that $\dim H^k(G,V) < C_k$ for  all absolutely
irreducible $FG$-modules $V$ and all finite simple groups $G$  with $F$ 
an algebraically closed field (of any characteristic)?
\end{question}

See \cite{cps} for some recent evidence related to this conjecture.

A slightly weaker version of this question for $k=1$ is relevant to
an old conjecture of Wall.   His conjecture is that the number of
maximal subgroups of a finite group $G$ is less than $|G|$.  If we
consider groups of the form $VH$ with $V$ an irreducible $\F_pH$-module,
a special case of Wall's conjecture (and likely the hardest case is):

\begin{question}  If $V$ is an irreducible $\F_pG$-module with $G$
finite, is $|H^1(G,V)| < |G|$?
\end{question}

This  is true for $G$ solvable \cite{wall},  and in that case is essentially
equivalent to Wall's conjecture.

We now give some examples to show that the analog of Theorem C
does not hold for $H^k, k > 2$.  

Let $F$ be an algebraically closed field of characteristic $p > 0$.

Let $S$ be a nonabelian finite simple group such that $p$ divides
both the order of $S$ and the order of its outer automorphism group.
Let $L$ be a subgroup of $\aut(S)$ containing $S$ with $L/S$
of order $p$.      Let $W$ be an irreducible $FS$-module
with $H^1(S,W) \ne 0$.  Note that if $x \in S$  has order $p$,
then all Jordan blocks of $x$ have size $p$ in
any projective $FS$-module. In particular, 
 the trivial module is not projective and so $H^1(S,W) \ne 0$
for some irreducible module $W$.  Obviously such a $W$ can not be the trivial
module.
   Let $U=W_S^L$.  Then either
$U$ is irreducible and $H^1(L,U) \cong H^1(S,W) \ne 0$ by Lemma \ref{shapiro},
or each of the $p$ composition factors of $U$ (as an $L$-module) is
 isomorphic to $W$
as  $FS$-modules.  Since $H^1(L,U) \ne 0$, some irreducible $L$-composition
factor of $U$ also has nontrivial $H^1$ by Lemma \ref{cohomology sequence}.
In either case, we see that there exists an irreducible faithful 
$FL$-module $V$
with $H^1(L,V) \ne 0$ and $H^1(L,F) \cong F$.

Let $G=L \wr C_t$ and let $N < G$ be the direct product
$L_1 \times \cdots \times L_t$ with $L_i \cong L$.   
Let $X=V \otimes F \cdots \otimes F$.  So
$X$ is an irreducible $FN$-module.   

By Lemma \ref{kunneth},  for $k \ge 3$, 
$$
\dim H^k(N,X) \ge \dim H^1(L,V) \cdot \binom{t-1}{k-1}
\ge c_kt^{k-1}.
$$
for some constant $c_k$.
   Thus,  for $t$ sufficiently large,   $\dim H^k(N,X) > d_kt^{k-1} \dim X$
   for the constant $d_k := (c_k \dim H^1(L,V))/(\dim V))$.
   
 Similarly, $\dim H^2(N,X) = \dim H^2(L,V) + (t-1) \dim H^1(L,V) \ge t-1$.

Now let $M = X_N^G$.   By Lemma \ref{shapiro}, 
$\dim H^k(G,M) = \dim H^k(N,X)$.

We record the following consequence for $k=2$.
 
 \begin{theorem} \label{example 2 lemma}  Let $F$ be an algebraically
 closed field of characteristic $p > 0$.  There is a constant $e_p > 0$
 such that if $d$ is a positive integer,  then there exist a finite group
 $G$ and an irreducible faithful $FG$-module with 
 $\dim H^2(G,M)  \ge e_p \dim M$ and $\dim M > d$.
 \end{theorem}

In particular, we see that $\dim H^2(G,M)$ can be arbitrarily
large for $M$ an irreducible faithful $FG$-module in any
characteristic.    Another way of stating the previous result is
that for a fixed $p$, 
$$
u(p):=\limsup_{\dim M \rightarrow \infty}  \frac{ \dim H^2(G,M) }{ \dim M } > 0.
$$
Here we are allowing any finite group $G$ with $M$ any irreducible faithful
$\F_pG$-module.  The scarce evidence suggests:

\begin{conjecture} $\lim_{p \rightarrow \infty} u(p)=0$.
\end{conjecture}

If $k > 2$, we obtain:

\begin{lemma}  \label{example lemma}  Keep notation as above.
\begin{enumerate}
\item $M$ is an irreducible faithful $\F_pG$-module with $\dim M = t \dim X$.
\item $\dim H^{k}(G,M) \ge d_kt^{k-2} \dim M$.  
\item There exists a constant $e_k >0 $ such
that $\dim H^k(G,M) \ge e_k(\dim M)^{k-1}$. 
\end{enumerate}
\end{lemma}

\begin{proof}  Note that $M$ is a direct sum of $t$ nonisomorphic irreducible
$FN$-modules that are permuted by $G$ and so $M$ is irreducible.  
Since $N$ is the unique minimal normal subgroup of $G$ and does not act 
trivially
on $M$, $G$ acts faithfully on $M$.  
Now (2) follows
by the discussion above and by Lemma \ref{shapiro}.  
Similarly, (3) follows with $e_k = c_k/(\dim V)^{k-1}$.  
\end{proof}

So we have shown:

\begin{theorem} \label{higher coh thm} Let $k$ be a positive integer. 
 If $k \ge 3$, there exist finite groups $G$ and faithful absolutely 
irreducible $FG$-modules
$M$ with $\dim H^k(G,M)/(\dim M)^{k-2}$ arbitrarily large.
\end{theorem}

 Our reduction methods in Section \ref{faithful}
   give very weak bounds for the dimension of $H^k(G,M)$
 with $M$ faithful and irreducible in terms of the bounds for
 the simple groups.
We ask whether our examples are the best possible:

\begin{question}   For which positive integers $k$ is it true that there is an absolute
constant $d_k$ such that $\dim H^k(G,V) < d_k(\dim V)^{k-1}$ for  
all absolutely
irreducible faithful $FG$-modules $V$ and all finite groups $G$  with $F$ 
an algebraically closed field (of any characteristic)?
\end{question}

For $k=1$, the question reduces to the case of simple groups.
Theorem C says that we can take $d_2=18.5$.

\section{Profinite Versus Discrete Presentations} \label{applications}

In this section, we consider discrete and profinite presentations for 
finite groups.
Recall that $r(G)$ (respectively $\hat{r}(G)$) denotes the minimal number of relations
required in a presentation (respectively profinite presentation) of a finite group $G$.
In fact, if $G=F/N$ is a discrete presentation of $G$ (i.e. $F$ is a free group), then 
$G=\hat{F}/\bar{N}$, where $\hat{F}$ is the profinite completion of $F$ and 
$\bar{N}$ is the closure of $N$ in $\hat{F}$.
So $\hat{F}/\bar{N}$ is a profinite presentation for $G$.
Indeed, every profinite presentation of $G$ can be obtained this way.

Let $R=N/[N,N]$ and for a prime $p$, set $R(p)=N/[N,N]N^{p}$.
So $R$ (resp. $R(p)$) is the relation (resp. $p$-relation) module of $G$
with respect to the given presentation. Denote by $d_F(N)$ the minimal number
of generators required for $N$ as a normal subgroup of $F$ and $d_G(R)$
(resp. $d_G(R(p))$
the minimal number of generators required for $R$ (resp. $R(p)$)
as a $\mathbb{Z}G$-module. Similarly,
define $\hat{d}_{\hat{F}}(\bar{N})$ to be the minimal number of generators
required for $\bar{N}$ as a closed normal subgroup of $\hat{F}$.

A theorem of Swan \cite[Theorem 7.8]{relmod} asserts that 
$d_G(R) = \max_p d_G(R(p))$,  and Lubotzky \cite{lub1} showed that
$\hat{d}_{\hat{F}}(\bar{N}) =  \max_p d_G(R(p))$.  
So altogether $\hat{d}_{\hat{F}}(\bar{N}) = d_G(R)$.  Moreover,
it is shown in \cite{lub1} that $\hat{r}(G)=\hat{d}_{\hat{F}}(\bar{N}) $
for any minimal presentation of $G$, i.e. a presentation in which $d(F)=d(G)$
(see also Lemma \ref{saving}).
    The analogous property
for discrete presentations of finite groups is not known 
and fails for infinite groups (cf. \cite[p. 2]{relmod}).

The long standing open problem whether $d_F(N)=d_G(R)$ (see \cite[p. 4]{relmod})
 therefore has
an equivalent formulation:

\begin{question} \label{all}   Is $d_F(N)=\hat{d}_{\hat{F}}(\bar{N})$?
\end{question}

A variant of this question is even more interesting:

\begin{question}  \label{relations} Is $\hat{r}(G)=r(G)$?
\end{question}

Of course, a positive answer to Question \ref{all} would imply a positive
answer  to Question \ref{relations},  but not conversely.

A weaker version of Question \ref{all} is:

\begin{question} Given a presentation $G = F/N = \hat{F}/\bar{N}$ of the finite
group $G$, are there $d_{\hat{F}}(\bar{N})=d_G(R)$ elements of $N$ which
generate $\bar{N}$ as a closed normal subgroup of $\hat{F}$?
\end{question}

In light of the above discussion, it is not surprising that our results
in \cite{pap1} and in the current paper give better estimates for
profinite presentations of finite simple groups than for  discrete presentations.
Theorem A ensures that all finite simple groups have profinite presentations
with at most $18$ relations.   We investigate discrete presentations
in \cite{pap1} and \cite{pap3}.  In  \cite{pap3}, we worry less about
the total length of relations and  prove:

\begin{theorem}  \label{bounded simple} Every finite simple group, 
with the possible exception
of ${^2}\G_2(3^{2k+1})$, has a presentation with $2$ generators and at most
$100$ relations.
\end{theorem}

Of course, both $18$ and $100$ are not optimal and indeed, as we
have already observed,  for many
groups we know much better bounds.  One may hope that $4$ is the right
upper bound for both types of presentations.  Indeed, there is no known
obstruction to the full covering group of a finite simple group having a
presentation with $2$ generators and $2$ relations  (see \cite{wilson}).

Let us now turn our attention to 
presentations (and cohomology) of general finite groups.

If $G = F/N$ is simple (and not ${^2}\G_2(q)$) with $F$
free and $d(F)=2$, 
then by the results of \cite{pap1}, $N$ can
be generated, as a normal subgroup of $F$, by $C$  words for some absolute 
constant $C$
(and the total length of the words used  can be bounded in
terms of $|G|$).   
Mann \cite{mann} showed that if every finite simple group can be presented
with $O(\log |G|)$ relations, then every finite group could be presented
with $O(d(G) \log |G|) \le O((\log|G|)^2)$ relations.      
Of course, in \cite{pap1},  we proved that simple
groups (with the possible exception of ${^2}\G_2(q)$) can be presented
with a bounded number of  relations -- but the better bound for simple groups 
does
not translate to a better bound for all groups.   Mann's argument  is valid also
in the profinite case and since there are no exceptions, we have:

\begin{theorem} \label{mann}  Let $G$ be a finite group.
\begin{enumerate} 
\item  If $G$  has no composition factors isomorphic to ${^2}G_2(3^{2k+1})$, 
then
$G$ has a presentation with  $O(d(G)\log|G|)$ relations.
\item  $G$ has a profinite presentation with $O(d(G)\log|G|)$ relations.
\end{enumerate}
\end{theorem}

The example of an elementary abelian $2$-group shows that one can do no
better in general.  Results like the above have been used to count groups of a given
order (or perfect groups of a given order) and also for getting
results on subgroup growth.   Fortunately the profinite result is sufficient
for these types of results and so the Ree groups do not cause problems.
See \cite{lub1}.  

Using the reduction of \cite[Theorem 1.4]{bgkl} 
to simple groups for lengths of presentations,
one sees that:

 \begin{theorem}  Let $G$ be any finite group with no composition 
 factors isomorphic to ${^2}\G_2(q)$.  Then $G$ has a presentation
 of length $O((\log|G|)^3)$.  
 \end{theorem}
 
 This is essentially in \cite{bgkl} aside from excluding $\SU(3,q)$
 and Suzuki groups (at the time of that paper it was not known that
 those groups had presentations with $\log |G|$ relations).
 As pointed out in \cite{bgkl},  the constant $3$ in the previous 
 theorem cannot be
 improved (by considering $2$-groups).

We now give some refinements of these results in the profinite setting.
We first prove some results about $H^2$.

We need to introduce some notation.    Recall that a chief factor $X$
of a finite group is a nontrivial section $A/B$ of $G$ where 
$B$ and $A$ are both
normal in $G$ and there is no normal subgroup of $G$ properly between $A$ and $B$.
Clearly $X$ is characteristically simple and so $X$ is either an elementary
abelian $r$-group for some prime $r$ or $X$ is isomorphic to a direct product
of copies of a nonabelian simple group and $G$ permutes these factors transitively.
There is an obvious definition of isomorphism of chief factors.  An
appropriate version of the  Jordan-H\"older
theorem   implies that the multi-set of chief factors coming from a maximal chain
of normal subgroups of $G$ is independent of the chain.

If $X$ is a nonabelian chief factor, let $s_p(X)$ denote the $p$-rank of
the Schur multiplier of a simple direct factor of $X$.  So $s_p(X) \le 2$
(and for $p > 3$, $s_p(X) \le 1$) \cite[pp. 312--313]{gls3}.   
Let $s_p(G)$ denote the sum of
the $s_p(X)$ as $X$ ranges over the nonabelian chief factors of $G$
(counting multiplicity).  
If $X$ is a chief factor of $G$ and is an elementary abelian $p$-group,  let
$\ell_p(X) = \log_p |X|$. Let $\ell_p(G)$ denote the sum of $\ell_p(X)$ as $X$
ranges over the chief factors of $G$ that are $p$-groups.

Define
$$
h_{p,1}(G) = \max\{1 + \dim H^1(G,V)/\dim V\},
$$
where $V$ is an irreducible $\F_pG$-module.  Note that this is always bounded
by $d(G)+1$ (or $d(G)$ if $V$ is nontrivial)  since a derivation
is determined by its images on a set of generators.  We can now prove:

\begin{theorem} \label{general h2}  Let $G$ be a finite group and $V$
an  $\F_pG$-module.   Then
$$ 
\dim H^2(G,V) \le (C + s_p(G) + h_{p,1}(G)\ell_p(G)) \dim V,
$$
where $C=18.5$ is the constant given in Theorem C.
\end{theorem}

\begin{proof}  Let $N$ be a minimal normal subgroup of $G$.
We first claim  that $h_{p,1}(G/N) \le h_{p,1}(G)$.
Let $W$ be an irreducible $\F_p(G/N)$-module which we may
consider as an $FG$-module.   Let $H=W.G$.   Then $N$ is normal
in $H$.  If $X$ is a complement to $W$ in $H/N$, then 
$Y$ is a complement to $W$ in $H$, where $Y$ is the preimage of
$X$ in $H$.  Thus, the number of complements of $W$ in
$H$ is at least the number of complements of $W$ in $H/N$.
So $\dim H^1(G/N,W)  \le \dim H^1(G,W)$, whence the claim.

It suffices to prove the theorem for $V$ irreducible.
If $G$ acts faithfully on $V$, this follows from Theorem C.
So we may assume that there is a minimal normal subgroup $N$ of $G$
that acts trivially on $V$.

By Lemma \ref{usual}, 
$$
\dim H^2(G,V) \le \dim H^2(G/N,V) + \dim H^2(N,V)^G + \dim H^1(G/N,H^1(N,V)).
$$

Suppose that $N$ is nonabelian.  Then $N$ is perfect, and so $H^1(N,V)=0$
by Lemma \ref{perfect}.
By Lemma \ref{kunneth}, $H^2(N,V) = \oplus H^2(L_i,V)$,
where $N$ is the direct product of the $L_i$.  Since $G$ permutes the
$L_i$ transitively,  it also permutes the $H^2(L_i,V)$, and so
 $H^2(N,V)^G$ embeds in $H^2(L,V)$ where $L \cong L_i$.
Since $V$ is a trivial module,  $\dim H^2(L,V) = \dim H^2(L,F) \dim V
=s_p(N)\dim V$.  So in this case, we have:
$\dim H^2(G,V) \le \dim H^2(G/N,V) + s_p(N)\dim V$ and the result
follows by induction.

Suppose that $N$ is abelian.  If $N$ is a $p'$-group, then the last two
terms in the inequality above are $0$ and the result follows.  So assume
that $N$ is an elementary abelian $p$-group.  Set $e=\ell_p(N)$.  By definition, 
$\dim  H^1(G/N,H^1(N,V)) \le (h_{p,1}(G)-1) e \dim V$.   
By induction, it suffices to show that  $\dim H^2(N,V)^G  \le e \dim V$.
By Lemma \ref{H2 for abelian}, 
$$
\dim H^2(N,V)^G \le \dim \mathrm{Hom}_G(N,V) +
\dim \mathrm{Hom}_G(\wedge^2(N),V).
$$   
If $V \cong N$, then clearly the number of composition factors
of $\wedge^2(N)$ isomorphic to $V$ is at most $(e-1)/2$,  and
so $\dim H^2(N,V)^G \le (e+1)/2 (\dim \mathrm{End_G(V)}) \le e \dim V$.
If $V$ is not isomorphic to $N$, then  $H^2(N,V)^G=0$, and so $\dim H^2(N,V)^G
\le \mathrm{Hom}_G(\wedge^2(N),V)$ and by Lemma \ref{exterior square}, 
$\dim H^2(N,V)^G  \le (e -1) \dim V$.   This completes the proof.
\end{proof}

Note that $s_p(G)$ is at most twice the number of nonabelian chief
factors of $G$.

If we only consider $d$-generated groups, then
as noted above, $h_{p,1} \le d+1$.  Indeed, $\dim H^1(G,V) \le (d-1) \dim V$
unless $V$ involves trivial modules.   So one has:

\begin{corollary} \label{d-general}  Let $G$ be a finite group
with $d(G) = d$  and $V$
an $\F_pG$-module.   Then $\dim H^2(G,V) \le (C + s_p(G) + (d+1)\ell_p(G)) \dim V$,
where $C=18.5$ is the constant given in Theorem C.
\end{corollary}

Now using (1.1), we can obtain results about profinite presentations.
Let $h_1(G)$ be the maximum of $h_{p,1}(G)$ over $p$, $\ell(G)$ the maximum
of the $\ell_p(G)$ and $s(G)$ the maximum of the $s_p(G)$.   The following
is a refinement of the results mentioned in the beginning of the section.

\begin{theorem} \label{gen-profinite}  Let $G$ be a finite group.
Then $\hat{r}(G) \le d(G) + C + s(G) + h_1(G)\ell(G)$, where $C - 1=18.5$ is the 
constant in Theorem C.   
In particular, if $d(G) \le d$, then $\hat{r}(G) \le d + C + s(G) + (d+1)\ell(G)$.
\end{theorem}

This improves Theorem \ref{mann} in the profinite setting  since $s(G)$ and $\ell(G)$
are  bounded above by  $\log_2 |G|$ and $h_1(G) \le d(G)+1$.

We mention some special cases that are a bit surprising.

\begin{corollary} \label{no abelian}  Let $G$ be a finite group with 
no abelian composition factors.  Then $\hat{r}(G) \le d(G) + 19 + 2s$
where $s$ is the number of chief factors of $G$.
\end{corollary}

\begin{corollary} \label{no schur}  Let $G$ be a finite group with 
no abelian composition factors and no composition factors that have
a nontrivial Schur multiplier.  Then $\hat{r}(G) \le d(G) + 19$.
\end{corollary}

It is not clear that the previous result is true for discrete presentations
and may suggest a strategy for proving that one does not always have
$r(G)=\hat{r}(G)$.

\end{document}